\documentclass[12pt]{amsart}
\usepackage{amssymb,amsmath,amsthm, amscd, enumerate, mathrsfs}
\usepackage{graphicx, hhline}
\usepackage[all]{xy}
\usepackage{braket}
\usepackage[usenames]{color}
\usepackage{hyperref}
\usepackage{fancyhdr}
\usepackage[top=30truemm,bottom=30truemm,left=30truemm,right=30truemm]{geometry}
\usepackage[pagewise]{lineno}

\hypersetup{colorlinks=true}

\title[Complex analytic MMP for normal pairs]{The minimal model program for normal pairs along the log canonical locus in the complex analytic setting}
\author{Kenta Hashizume}
\date{2026/07/25}
\keywords{minimal model theory, normal pair, complex analytic space}
\subjclass[2020]{14E30}
\address{Department of 
Mathematics, Faculty of Science, Niigata University, Niigata 950-2181, Japan}
\email{hkenta@math.sc.niigata-u.ac.jp}

\newtheorem{thm}{Theorem}[section]

\newtheorem{lem}[thm]{Lemma}
\newtheorem{cor}[thm]{Corollary}

\theoremstyle{definition}
\newtheorem{defn}[thm]{Definition}
\newtheorem{rem}[thm]{Remark}

\newtheorem*{ack}{Acknowledgments} 
\newtheorem*{divisor}{Divisors and morphisms} 
\newtheorem*{b-divisor}{b-divisors} 
\newtheorem*{Rlinearsystem}{$\mathbb{R}$-linear systems}
\newtheorem*{nonneflocus}{Non-nef locus} 
\newtheorem*{quasi-log}{Quasi-log complex analytic spaces} 
\newtheorem*{sing}{Singularities of pairs} 
\newtheorem*{g-pair}{Generalized pairs} 
\newtheorem*{adj-g-pair}{Divisorial adjunction for generalized pairs} 
\newtheorem*{mmp-g-pair}{MMP for generalized pairs}

\newtheorem{step5}{Step}

\newtheorem*{claim*}{Claim}

\begin{document}

\begin{abstract}
We establish the minimal model theory for normal pairs along the log canonical locus in the complex analytic setting. 
This is the complex analytic analog of the previous result by the author \cite{has-mmp-normal-pair}. 
\end{abstract}

\maketitle

\tableofcontents

\section{Introduction}

After the paper \cite{fujino-analytic-bchm} by Fujino, the minimal model theory for log canonical (lc, for short) pairs in the complex analytic setting is being developed (\cite{fujino-analytic-conethm}, \cite{eh-analytic-mmp}, \cite{fujino-analytic-lcabundance}, \cite{eh-analytic-mmp-2}). 
Moreover, Fujino \cite{fujino-quasi-log-analytic} established the foundations of {\em quasi-log complex analytic spaces}, which is the complex analytic analog of quasi-log schemes. 
In the algebraic case, the theory of quasi-log schemes played an important role in \cite{has-mmp-normal-pair} to prove the minimal model program (MMP, for short) for normal pairs whose singularities are beyond lc. 
In the paper \cite{has-mmp-normal-pair}, a lot of results and ideas in \cite{bchm}, \cite{birkar-flip}, \cite{haconxu-lcc},  \cite{hashizumehu}, and \cite{fujino-book} were involved. 
Currently, it is known that those results and ideas are valid for the complex analytic setting (\cite{fujino-analytic-lcabundance}, \cite{eh-analytic-mmp}, \cite{eh-analytic-mmp-2}, \cite{fujino-quasi-log-analytic}). 
Therefore, it is natural to study an analogous result of \cite{has-mmp-normal-pair} in the complex analytic setting.

The goal of this paper is to study the termination of the MMP for normal pairs. 
Following \cite{fujino-analytic-bchm} by Fujino, we deal with projective morphisms $\pi \colon X \to Z$ of complex analytic spaces and a compact subset $W \subset Z$ such that $\pi$ and $W$ satisfy the property (P) as in Definition \ref{defn--property(P)}. 
Because we are interested in the existence of minimal models or Mori fiber spaces for normal (or lc) pairs over an open neighborhood of $W$, we will freely shrink $Z$ around $W$ if necessary.

The following theorem is the main result of this paper.

\begin{thm}[cf.~Theorem \ref{thm--termination-lcmmp-main}]\label{thm--termination-lcmmp-intro}
Let $\pi \colon X \to Z$ be a projective morphism from a normal analytic variety $X$ to a Stein space $Z$, and let $W \subset Z$ be a compact subset such that $\pi$ and $W$ satisfy (P). 
Let $(X,\Delta)$ be a normal pair and $A$ a $\pi$-ample $\mathbb{R}$-divisor on $X$ such that $K_{X}+\Delta+A$ is globally $\mathbb{R}$-Cartier and $\pi$-pseudo-effective. 
Suppose that ${\rm NNef}(K_{X}+\Delta+A/Z) \cap {\rm Nlc}(X,\Delta) \cap \pi^{-1}(W)  = \emptyset$. 
Suppose in addition that $(K_{X}+\Delta+A)|_{{\rm Nlc}(X,\Delta)}$, which we think of as an $\mathbb{R}$-line bundle on ${\rm Nlc}(X,\Delta)$, is semi-ample over a neighborhood of $W$. 
We put $(X_{1},B_{1}):=(X,\Delta+A)$. 
Then, after shrinking $Z$ around $W$, there exists a diagram
$$
\xymatrix{
(X_{1},B_{1})\ar@{-->}[r]&\cdots \ar@{-->}[r]& (X_{i},B_{i})\ar@{-->}[rr]\ar[dr]_{\varphi_{i}}&& (X_{i+1},B_{i+1})\ar[dl]^{\varphi'_{i}} \ar@{-->}[r]&\cdots\ar@{-->}[r]&(X_{m},B_{m}) \\
&&&V_{i}
}
$$
satisfying the following.
\begin{itemize}
\item
For each $1 \leq i <m$, the diagram $(X_{i},B_{i}) \overset{\varphi_{i}}{\to} V_{i} \overset{\varphi'_{i}}{\leftarrow} (X_{i+1},B_{i+1})$ is the standard step of a $(K_{X_{1}}+B_{1})$-MMP over $Z$ around $W$, that is, $\varphi_{i}$ is bimeromorphic, $\rho(X_{i}/Z; W) - \rho(V_{i}/Z; W)=1$, $\varphi'_{i}$ is small bimeromorphic, and $-(K_{X_{i}}+B_{i})$ and $K_{X_{i+1}}+B_{i+1}$ are ample over $V_{i}$, 
\item
the MMP is represented by bimeromorphic contractions (see Definition \ref{def--mmp-represent-bimerocont}), 
\item
the non-biholomorphic locus of the MMP is disjoint from ${\rm Nlc}(X,\Delta)$, and
\item
$K_{X_{m}}+B_{m}$ is semi-ample over $Z$. 
\end{itemize}
Moreover, if $X$ is $\mathbb{Q}$-factorial over $W$, then all $X_{i}$ are also $\mathbb{Q}$-factorial over $W$. 
\end{thm}

The following results are corollaries of the main result. 

\begin{cor}[cf.~{\cite[Corollary 5.5]{has-mmp-normal-pair}}]\label{cor--nonvan-intro}
Let $\pi \colon X \to Z$ be a projective morphism from a normal analytic variety $X$ to a Stein space $Z$, and let $W \subset Z$ be a compact subset such that $\pi$ and $W$ satisfy (P). 
Let $(X,\Delta)$ be a normal pair and $A$ a $\pi$-ample $\mathbb{R}$-divisor on $X$ such that $K_{X}+\Delta+A$ is globally $\mathbb{R}$-Cartier and $\pi$-pseudo-effective. 
Suppose that ${\rm NNef}(K_{X}+\Delta+A/Z) \cap {\rm Nlc}(X,\Delta) \cap \pi^{-1}(W)  = \emptyset$. 
Suppose in addition that $(K_{X}+\Delta+A)|_{{\rm Nlc}(X,\Delta)}$, which we think of as an $\mathbb{R}$-line bundle on ${\rm Nlc}(X,\Delta)$, is semi-ample over a neighborhood of $W$. 
Then we have ${\rm Bs}|K_{X}+\Delta+A/Z|_{\mathbb{R}} \cap {\rm Nlc}(X,\Delta)\cap \pi^{-1}(W) = \emptyset$. 
\end{cor}

\begin{cor}[cf.~{\cite[Corollary 5.6]{has-mmp-normal-pair}}]\label{cor--mmp-from-nonvan-intro}
Let $\pi \colon X \to Z$ be a projective morphism from a normal analytic variety $X$ to a Stein space $Z$, and let $W \subset Z$ be a compact subset such that $\pi$ and $W$ satisfy (P). 
Let $(X,\Delta)$ be a normal pair and $A$ a $\pi$-ample $\mathbb{R}$-divisor on $X$ such that $K_{X}+\Delta+A$ is globally $\mathbb{R}$-Cartier and $\pi$-pseudo-effective. 
Suppose that ${\rm Bs}|K_{X}+\Delta+A/Z|_{\mathbb{R}} \cap {\rm Nlc}(X,\Delta) \cap \pi^{-1}(W)  = \emptyset$. 
We put $(X_{1},B_{1}):=(X,\Delta+A)$. 
Let $H$ be a $\pi$-ample $\mathbb{R}$-divisor on $X$. 
Then, after shrinking $Z$ around $W$, there exists a sequence of steps of a $(K_{X_{1}}+B_{1})$-MMP over $Z$ around $W$ with scaling of $H$ 
$$
\xymatrix{
(X_{1},B_{1})\ar@{-->}[r]&\cdots \ar@{-->}[r]& (X_{i},B_{i})\ar@{-->}[rr]\ar[dr]&& (X_{i+1},B_{i+1})\ar[dl] \ar@{-->}[r]&\cdots\ar@{-->}[r]&(X_{m},B_{m}) \\
&&&V_{i}
}
$$
represented by bimeromorphic contractions, where $(X_{i},B_{i}) \to V_{i} \leftarrow (X_{i+1},B_{i+1})$ is a step of a $(K_{X_{i}}+B_{i})$-MMP over $Z$ around $W$, such that
\begin{itemize}
\item
the non-biholomorphic locus of the MMP is disjoint from ${\rm Nlc}(X,\Delta)$, 
\item
$\rho(X_{i}/Z;W)-\rho(V_{i}/Z;W)=1$ for every $i \geq 1$, and
\item
$K_{X_{m}}+B_{m}$ is semi-ample over $Z$. 
\end{itemize}
Moreover, if $X$ is $\mathbb{Q}$-factorial over $W$, then all $X_{i}$ are also $\mathbb{Q}$-factorial over $W$. 
\end{cor}

\begin{cor}[cf.~{\cite[Corollary 5.7]{has-mmp-normal-pair}}]\label{cor--mmp-lc-strictnefthreshold-intro}
Let $\pi \colon X \to Z$ be a projective morphism from a normal analytic variety $X$ to a Stein space $Z$, and let $W \subset Z$ be a compact subset such that $\pi$ and $W$ satisfy (P). 
Let $(X,\Delta)$ be a normal pair such that $K_{X}+\Delta$ is $\pi$-pseudo-effective and ${\rm NNef}(K_{X}+\Delta/Z) \cap {\rm Nlc}(X,\Delta) \cap \pi^{-1}(W)  = \emptyset$. 
Let $A$ be a $\pi$-ample $\mathbb{R}$-divisor on $X$. 
Then there exists a sequence of steps of a $(K_{X}+\Delta)$-MMP over $Z$ around $W$ with scaling of $A$ 
$$
\xymatrix{
(X_{1},\Delta_{1})\ar@{-->}[r]&\cdots \ar@{-->}[r]& (X_{i},\Delta_{i})\ar@{-->}[rr]\ar[dr]&& (X_{i+1},\Delta_{i+1})\ar[dl] \ar@{-->}[r]&\cdots, \\
&&&V_{i}
}
$$
where $(X_{i},\Delta_{i}) \to V_{i} \leftarrow (X_{i+1},\Delta_{i+1})$ is a step of the $(K_{X}+\Delta)$-MMP over $Z$ around $W$, such that
\begin{itemize}
\item
the non-biholomorphic locus of the MMP is disjoint from ${\rm Nlc}(X,\Delta)$, 
\item
$\rho(X_{i}/Z;W)-\rho(V_{i}/Z;W)=1$ for every $i \geq 1$, and
\item
if we put 
$$\lambda_{i}:={\rm inf}\{\mu \in \mathbb{R}_{\geq 0}\,| \,\text{$K_{X_{i}}+\Delta_{i}+\mu A_{i}$ is nef over $W$} \}$$  
for each $i \geq 1$, then ${\rm lim}_{i \to \infty}\lambda_{i}=0$. 
\end{itemize}
Moreover, if $X$ is $\mathbb{Q}$-factorial over $W$, then all $X_{i}$ are also $\mathbb{Q}$-factorial over $W$. 
\end{cor}

\begin{cor}[cf.~{\cite[Corollary 5.10]{has-mmp-normal-pair}}]\label{cor--finite-generation-intro}
Let $\pi \colon X \to Z$ be a projective morphism from a normal analytic variety to a Stein space. 
Let $(X,\Delta)$ be a normal pair such that $\Delta$ is a $\mathbb{Q}$-divisor on $X$. 
Let $A$ be a $\pi$-ample $\mathbb{Q}$-divisor on $X$ such that $K_{X}+\Delta+A$ is globally $\mathbb{Q}$-Cartier and $\pi$-pseudo-effective and ${\rm NNef}(K_{X}+\Delta+A/Z) \cap {\rm Nlc}(X,\Delta) = \emptyset.$
 Suppose in addition that $(K_{X}+\Delta+A)|_{{\rm Nlc}(X,\Delta)}$, which we think of as a $\mathbb{Q}$-line bundle on ${\rm Nlc}(X,\Delta)$, is semi-ample over $Z$. 
Then the sheaf of graded $\pi_{*}\mathcal{O}_{X}$-algebras
$$\underset{m \in \mathbb{Z}_{\geq 0}}{\bigoplus}\pi_{*}\mathcal{O}_{X}(\lfloor m(K_{X}+\Delta+A)\rfloor)$$
is locally finitely generated. 
\end{cor}

The proofs of the main result and the corollaries are almost the same as in the algebraic case (\cite{has-mmp-normal-pair}). 
We recommend reading \cite{has-mmp-normal-pair} to see the circle of ideas and the detailed arguments.
In this paper, we mainly write the definitions used in the proofs and the detailed proof of one of the key ingredients (Theorem \ref{thm--mmpstep-quasi-log}) that the argument in \cite{has-mmp-normal-pair} does not work well because of lack of technical results.

The contents of this paper are as follows: In Section \ref{sec--preliminaries}, we collect some definitions. 
In Section \ref{sec--running-mmp}, we prove results to run an MMP for normal pairs. 
In Section \ref{sec--mmp-lc}, we prove the main result and corollaries. 
In Section \ref{sec--glue}, we discuss the MMP for normal pairs for algebraic stacks and analytic stacks. 

\begin{ack}
The author was partially supported by JSPS KAKENHI Grant Number JP22K13887. 
The author is grateful to Makoto Enokizono for helpful discussions.
The author is also grateful to the referees for carefully reading the manuscript and for many valuable suggestions, which greatly improved the paper. 
\end{ack}

\section{Preliminaries}\label{sec--preliminaries}

Throughout this paper, {\em complex analytic spaces} are always assumed to be Hausdorff and second countable. 
{\em Analytic varieties} are reduced and irreducible complex analytic spaces. 
For notations and definitions used in this paper, see \cite{fujino-analytic-bchm} and  \cite{eh-analytic-mmp}. 

For the algebraic analogs of the results in this section, see \cite[Section 2]{has-mmp-normal-pair}. 

\begin{divisor}
Let $\pi \colon X \to Z$ be a projective morphism from a normal analytic variety to a complex analytic space. 
We will freely use the definitions of $\pi$-nef $\mathbb{R}$-divisor, $\pi$-ample $\mathbb{R}$-divisor, $\pi$-semi-ample $\mathbb{R}$-divisor, and $\pi$-pseudo-effective $\mathbb{R}$-Cartier divisor as in \cite[Subsection 2.44]{fujino-analytic-bchm}. 
The real vector space spanned by the prime divisors on $X$ is denoted by ${\rm WDiv}_{\mathbb{R}}(X)$ (\cite[Definition 2.34]{fujino-analytic-bchm}). 
For a prime divisor $P$ over $X$, the image of $P$ on $X$ is denoted by $c_{X}(P)$. 

A {\em contraction} $f\colon X\to Y$ is a projective morphism of analytic varieties such that $f_{*}\mathcal{O}_{X}\cong \mathcal{O}_{Y}$. 
For an analytic variety $X$ and an $\mathbb{R}$-divisor $D$ on $X$, a {\em log resolution of} $(X,D)$ is a proper bimeromorphic morphism $g \colon V\to X$ from a smooth analytic variety $V$ such that $g$ is projective over a neighborhood of any compact subset of $X$, the exceptional locus ${\rm Ex}(g)$ of $g$ is of pure codimension one and ${\rm Ex}(g)\cup {\rm Supp}\,g_{*}^{-1}D$ is a simple normal crossing divisor. 

We refer to \cite[Definition 2.30]{fujino-analytic-bchm} for {\em meromorphic map} and {\em bimeromorphic map}. 
A bimeromorphic map $\phi\colon X \dashrightarrow X'$ of analytic varieties is called a {\em bimeromorphic contraction} if $\phi^{-1}$ does not contract any divisor. 
We say that $\phi$ is {\em small} if $\phi$ and $\phi^{-1}$ are bimeromorphic contractions.

\begin{defn}[Property (P), see {\cite{fujino-analytic-bchm}}]\label{defn--property(P)}
Let $\pi \colon X \to Y$ be a projective morphism of complex analytic spaces, and let $W \subset Y$ be a compact subset. 
In this paper, we will use the following conditions:
\begin{itemize}
\item[(P1)]
$X$ is a normal analytic variety,
\item[(P2)]
$Y$ is a Stein space,
\item[(P3)]
$W$ is a Stein compact subset of $Y$, and
\item[(P4)]
$W \cap Z$ has only finitely many connected components for any analytic subset $Z$
which is defined over an open neighborhood of $W$. 
\end{itemize}
We say that {\em $\pi \colon X \to Y$ and $W \subset Y$ satisfy (P)} if the conditions (P1)--(P4) hold. 
\end{defn}

\begin{defn}[$\mathbb{R}$-line bundle, relative ampleness, relative semi-ampleness]\label{defn--R-line-bundle-an-sp}
Let $X$ be a complex analytic space and let ${\rm Pic}(X)$ be the Picard group of $X$. 
A {\em $\mathbb{Q}$-line bundle on $X$} is an element of ${\rm Pic}(X)\otimes_{\mathbb{Z}}\mathbb{Q}$. 
Let $\pi \colon X \to Z$ be a projective morphism to a complex analytic space $Z$. 
A $\mathbb{Q}$-line bundle $\mathcal{L}$ on $X$ is {\em $\pi$-ample} or {\em ample over $Z$} if $\mathcal{L}$ is a finite $\mathbb{Q}_{>0}$-linear combination of $\pi$-ample invertible sheaves on $X$. 
A $\mathbb{Q}$-line bundle $\mathcal{L}$ on $X$ is {\em $\pi$-semi-ample} or {\em semi-ample over $Z$} if we can write $\mathcal{L}=\sum_{i=1}^{k}q_{i} \mathcal{L}_{i}$ as an element of ${\rm Pic}(X)\otimes_{\mathbb{Z}}\mathbb{Q}$ such that $q_{1},\,\cdots,\,q_{k}$ are positive rational numbers and $\mathcal{L}_{1},\,\cdots,\,\mathcal{L}_{k}$ are globally generated over $Z$, in other words, the natural morphism
$\pi^{*}\pi_{*}\mathcal{L}_{i} \to \mathcal{L}_{i}$ 
is surjective for all $1 \leq i \leq k$. 

Similarly, we define an {\em $\mathbb{R}$-line bundle on $X$} to be an element of ${\rm Pic}(X)\otimes_{\mathbb{Z}}\mathbb{R}$, and we say that an $\mathbb{R}$-line bundle $\mathcal{L}$ on $X$ is {\em $\pi$-ample} or {\em ample over $Z$} (resp.~{\em $\pi$-semi-ample} or {\em semi-ample over $Z$}) if $\mathcal{L}$ is a finite $\mathbb{R}_{>0}$-linear combination of invertible sheaves on $X$ that are $\pi$-ample (resp.~globally generated over $Z$). 
\end{defn}
\end{divisor}

\begin{Rlinearsystem}\label{subsec--linear-system}

In this subsection, we define the relative $\mathbb{R}$-linear system and the relative stable base locus of an $\mathbb{R}$-divisor.

\begin{defn}[Relative $\mathbb{R}$-linear system, relative stable base locus]\label{defn--R-linear-system}
Let $\pi \colon X \to Z$ be a projective morphism from a normal analytic variety $X$ to a Stein space $Z$, and we fix $\mathbb{F} \in \{\mathbb{Q},\,\mathbb{R}\}$.  
Let $D$ be a (not necessarily $\mathbb{F}$-Cartier) $\mathbb{F}$-divisor on $X$. 
Then the {\em $\mathbb{F}$-linear system of $D$}, denoted by $|D|_{\mathbb{F}}$, is defined by
$$|D|_{\mathbb{F}}:=\{E \geq 0\,|\, E \sim_{\mathbb{F}}D\}.$$ 
The {\em $\mathbb{F}$-linear system of $D$ over $Z$}, denoted by $|D/Z|_{\mathbb{F}}$, is defined by
$$|D/Z|_{\mathbb{F}}:=\{E \geq 0\,|\, E \sim_{\mathbb{F},\, Z}D\}.$$ 
The {\em stable base locus of $D$ over $Z$}, denoted by ${\rm Bs}|D/Z|_{\mathbb{F}}$, is defined by
$${\rm Bs}|D/Z|_{\mathbb{F}}:=\bigcap_{E \in |D/Z|_{\mathbb{F}}}{\rm Supp}\,E.$$ 
If $|D/Z|_{\mathbb{F}}$ is empty, then we set ${\rm Bs}|D/Z|_{\mathbb{F}}=X$ by convention. 

We say that an $\mathbb{R}$-divisor $D''$ on $X$ is {\em movable over $Z$} if the codimension of ${\rm Bs}|D''/Z|_{\mathbb{R}}$ in $X$ is greater than or equal to two.
\end{defn}

\begin{rem}[{\cite[Proposition 2.9]{eh-analytic-mmp}}]\label{rem--stable-base-locus-restriction}
Let $\pi \colon X \to Z$ and $D$ be as in Definition \ref{defn--R-linear-system} and $U \subset Z$ a Stein open subset. 
If $D$ is a globally $\mathbb{F}$-Cartier divisor, then we have
$$({\rm Bs}|D/Z|_{\mathbb{F}})|_{\pi^{-1}(U)}={\rm Bs}|D|_{\pi^{-1}(U)}/U|_{\mathbb{F}}$$
as subsets of $\pi^{-1}(U)$. 
\end{rem}

\begin{rem}
The complex analytic analogs of \cite[Lemmas 2.4, 2.5, Theorem 2.6]{has-mmp-normal-pair} hold because the arguments in the algebraic case work in our situation.  
\end{rem}

\end{Rlinearsystem}

\begin{nonneflocus}\label{subsec--non-nef-locus}

The goal of this subsection is to define the relative non-nef locus.

\begin{defn}[Asymptotic vanishing order, {\cite[III, \S 4.a]{nakayama}}, {\cite[Subsection 4.1]{eh-analytic-mmp}}]\label{defn--asymp-van-order}
Let $\pi \colon X \to Z$ be a projective morphism from a normal analytic
variety $X$ to a Stein space $Z$.
Let $D$ be a $\pi$-pseudo-effective $\mathbb{R}$-Cartier divisor on $X$. 
For a prime divisor $P$ over $X$, we define the {\em asymptotic vanishing order of $D$ along $P$ over $Z$}, denoted by $\sigma_{P}(D; X/Z)$, as follows: 
We take a resolution $f \colon X' \to X$ of $X$ on which $P$ appears as a prime divisor. 
We put 
$$m_{f^{*}D}:={\rm max}\left\{m \in \mathbb{Z}_{\geq 0}\,\middle|\begin{array}{l}(\pi\circ f)_{*}\mathcal{O}_{X'}(\lfloor f^{*}D \rfloor-mP)\hookrightarrow(\pi\circ f)_{*}\mathcal{O}_{X'}(\lfloor f^{*}D \rfloor)\\ 
\text{is an isomorphism}\end{array}\right\}$$
if $(\pi\circ f)_{*}\mathcal{O}_{X'}(\lfloor f^{*}D \rfloor)$ is not the zero sheaf. 
When $D$ is big over $Z$, then
$$\sigma_{P}(f^{*}D; X'/Z)_{\mathbb{Z}}:=\left\{\begin{array}{ll}
+\infty& ((\pi\circ f)_{*}\mathcal{O}_{X'}(\lfloor f^{*}D \rfloor)=0)\\
m_{f^{*}D}+{\rm coeff}_{P}(\{f^{*}D\})& ((\pi\circ f)_{*}\mathcal{O}_{X'}(\lfloor f^{*}D \rfloor)\neq0)
\end{array}\right.$$
and $\sigma_{P}(D; X/Z)$ is defined to be
$$\sigma_{P}(D; X/Z):=\underset{m \to \infty}{\rm lim}\frac{1}{m}\sigma_{P}(mf^{*}D; X'/Z)_{\mathbb{Z}}.$$
Note that $\sigma_{P}(D; X/Z)$ is independent of $f$. 
If $D$ is not necessarily $\pi$-big, then we take a $\pi$-ample $\mathbb{R}$-divisor $A$ on $X$, and we set
$$\sigma_{P}(D; X/Z):=\underset{\epsilon\to 0+}{\rm lim}\sigma_{P}(D+\epsilon A; X/Z).$$
Note that $\sigma_{P}(D; X/Z)$ does not depend on $A$, and we may have $\sigma_{P}(D; X/Z)=+\infty$ for some $P$. 
In fact, if $D$ is $\pi$-big and $f \colon X' \to X$ is a resolution of $X$ on which $P$ appears as a prime divisor, then \cite[III, 4.1 Lemma]{nakayama} and \cite[Lemma 4.3]{eh-analytic-mmp} imply that after shrinking $Z$ around a general point of $(\pi\circ f)(P)$, we have
$$\sigma_{P}(D; X/Z)={\rm inf}\set{{\rm coeff}_{P}(E)| E \in |f^{*}D/Z|_{\mathbb{R}}}={\rm inf}\set{{\rm coeff}_{P}(E)| E \in |f^{*}D|_{\mathbb{R}}}.$$
\end{defn}

\begin{defn}[Relative non-nef locus, cf.~{\cite[III, 2.6 Definition]{nakayama}}]\label{defn--nonneflocus}
Let $\pi \colon X \to Z$ be a projective morphism from a normal analytic
variety $X$ to a Stein space $Z$. 
Let $D$ be a $\pi$-pseudo-effective $\mathbb{R}$-Cartier divisor on $X$. 
We define the {\em non-nef locus of $D$ over $Z$}, denoted by ${\rm NNef}(D/Z)$, to be
$${\rm NNef}(D/Z):=\underset{\sigma_{P}(D; X/Z)>0}{\bigcup}c_{X}(P)$$
where $P$ runs over prime divisors over $X$ and $\sigma_{P}(D; X/Z)$ is the asymptotic vanishing order of $D$ along $P$ over $Z$ in Definition \ref{defn--asymp-van-order}.   
\end{defn}

\begin{rem}\label{rem--nonnef-restriction}
Let $\pi \colon X \to Z$ and $D$ be as in Definition \ref{defn--nonneflocus} and $U \subset Z$ a Stein open subset. 
By \cite[Theorem 4.7 (6)]{eh-analytic-mmp}, we have
$${\rm NNef}(D/Z)|_{\pi^{-1}(U)}={\rm NNef}(D|_{\pi^{-1}(U)}/U)$$
as subsets of $\pi^{-1}(U)$. 
\end{rem}

\begin{rem}[{cf.~\cite{bbp}, \cite{tsakanikas-xie}, \cite[Remark 2.13, Lemma 2.15]{has-mmp-normal-pair}}]\label{rem--nonnef-relation}
We have
$${\rm NNef}(D/Z) \subset {\rm Bs}|D/Z|_{\mathbb{R}}$$
for every $\pi \colon X \to Z$ and $D$ as in Definition \ref{defn--nonneflocus}. 
Furthermore, we have 
$${\rm NNef}(D/Z) \supset f({\rm NNef}(f^{*}D/Z))$$ 
for any projective surjective morphism $f \colon Y \to X$ from a normal analytic variety $Y$, and the two sides coincide when $f$ is bimeromorphic.  
\end{rem}

\begin{thm}[cf.~{\cite[Theorem 2.14]{has-mmp-normal-pair}}]\label{thm--nonnef-negativecurve}
Let $\pi \colon X \to Z$ be a projective morphism from a normal analytic
variety $X$ to a Stein space $Z$. 
Let $D$ be a $\pi$-pseudo-effective $\mathbb{R}$-Cartier divisor on $X$. 
Let $C$ be a curve on $X$ such that $\pi(C)$ is a point and $C \not\subset {\rm NNef}(D/Z)$. 
Then $(D \,\cdot\, C) \geq 0$. 
\end{thm}

\begin{proof}
By Remark \ref{rem--nonnef-restriction}, we may freely shrink $Z$ around $\pi(C)$ without loss of generality. 
Then the argument from the proof of \cite[Theorem 2.14]{has-mmp-normal-pair} works with no changes. 
\end{proof}

\end{nonneflocus}

\begin{sing}

In this subsection, we collect some definitions and basic results on singularities of pairs.

A {\em normal sub-pair} $(X,\Delta)$ consists of a normal analytic variety $X$ and an $\mathbb{R}$-divisor $\Delta$ on $X$ such that $K_{X}+\Delta$ is $\mathbb{R}$-Cartier. 
A {\em normal pair} is a normal sub-pair $(X,\Delta)$ such that $\Delta$ is effective. 
We use these terms to explicitly state the normality of $X$. 
For a normal sub-pair $(X,\Delta)$ and a prime divisor $P$ over $X$, the {\em discrepancy of $P$ with respect to $(X,\Delta)$} is denoted by $a(P,X,\Delta)$. 
A normal pair $(X,\Delta)$ is {\em Kawamata log terminal} ({\em klt}, for short) if $a(P,X,\Delta)>-1$ for all prime divisors $P$ over $X$. 
A normal pair $(X,\Delta)$ is {\em log canonical} ({\em lc}, for short) if $a(P,X,\Delta) \geq -1$ for all prime divisors $P$ over $X$. 
A normal pair $(X,\Delta)$ is {\em divisorially log terminal} ({\em dlt}, for short) if $(X,\Delta)$ is lc and there exists a log resolution $f \colon Y \to X$ of $(X,\Delta)$ such that every $f$-exceptional prime divisor $E$ on $Y$ satisfies $a(E,X,\Delta) > -1$. 

\begin{defn}[Non-lc locus, non-klt locus]\label{defn--nlc-nklt}
Let $(X,\Delta)$ be a normal pair. 
We define the {\em non-lc locus} and the {\em non-klt locus} of $(X,\Delta)$, denoted by ${\rm Nlc}(X,\Delta)$ and ${\rm Nklt}(X,\Delta)$ respectively, to be complex analytic subspaces of $X$ by the following construction: 
We take a log resolution $f \colon Y \to X$ of $(X,\Delta)$. 
We may write $K_{Y}+\Gamma = f^{*}(K_{X}+\Delta)$ with an $\mathbb{R}$-divisor $\Gamma$ on $Y$. 
Then the natural isomorphism $\mathcal{O}_{X} \to f_{*}\mathcal{O}_{Y}(\lceil -(\Gamma^{<1})\rceil)$
defines ideal sheaves
\begin{equation*}
\begin{split}
&\mathcal{I}_{{\rm Nlc}(X,\Delta)}:=f_{*}\mathcal{O}_{Y}(\lceil -(\Gamma^{<1})\rceil-\lfloor \Gamma^{>1}\rfloor)=f_{*}\mathcal{O}_{Y}(-\lfloor \Gamma\rfloor+\Gamma^{=1}),\quad {\rm and}\\
&\mathcal{I}_{{\rm Nklt}(X,\Delta)}:=f_{*}\mathcal{O}_{Y}(-\lfloor \Gamma\rfloor)=f_{*}\mathcal{O}_{Y}(\lceil -(\Gamma^{<1})\rceil-\lfloor \Gamma^{>1}\rfloor-\Gamma^{=1})
\end{split}
\end{equation*}
of $\mathcal{O}_{X}$. 
These ideal sheaves are independent of the log resolution $f \colon Y \to X$ (\cite{fst-suppli-nonlc-ideal}). 
Then ${\rm Nlc}(X,\Delta)$ and ${\rm Nklt}(X,\Delta)$ are the complex analytic subspaces of $X$ defined by $\mathcal{I}_{{\rm Nlc}(X,\Delta)}$ and $\mathcal{I}_{{\rm Nklt}(X,\Delta)}$, respectively. 
\end{defn}

${\rm Nlc}(X,\Delta)$ and ${\rm Nklt}(X,\Delta)$ sometimes mean the support of the non-lc locus and the non-klt locus of $(X,\Delta)$ respectively if there is no risk of confusion.

\begin{defn}[Lc center]\label{defn--lccenter}
Let $(X,\Delta)$ be a normal pair. 
A subset $S \subset X$ is called an {\em lc center} of $(X,\Delta)$ if $S \not\subset {\rm Nlc}(X,\Delta)$ and there exists a prime divisor $P$ over $X$ such that $S=c_{X}(P)$ and $a(P,X,\Delta)=-1$. 
Unless otherwise stated, we assume lc centers to be reduced complex analytic subspaces of $X$. 
\end{defn}

The following result by Fujino \cite{fujino-analytic-bchm} is implicitly used in the proof of the main result of this paper.

\begin{thm}[Dlt blow-up, {\cite[Theorem 1.27]{fujino-analytic-bchm}}]\label{thm--dlt-blowup}
Let $\pi \colon X \to Z$ be a projective morphism from a normal analytic variety $X$ to a complex analytic space $Z$, and let $W \subset Z$ be a compact subset such that $\pi$ and $W$ satisfy (P). 
Let $(X,\Delta)$ be a normal pair. 
Then, after shrinking $Z$ around $W$, we can construct a projective bimeromorphic morphism $f \colon Y \to X$ from a normal complex variety $Y$ with the following properties. 
\begin{itemize}
\item
$Y$ is $\mathbb{Q}$-factorial over $W$,
\item
$a(E,X,\Delta)\leq -1$ for every $f$-exceptional prime divisor $E$ on $Y$, and
\item
$(Y,\Delta^{<1}_{Y}+{\rm Supp}\,\Delta^{\geq 1}_{Y})$ is dlt, where $K_{Y}+\Delta_{Y}=f^{*}(K_{X}+\Delta)$.
\end{itemize}
We call the morphism $f\colon (Y,\Delta_{Y}) \to (X,\Delta)$ a dlt blow-up. 
\end{thm}

\begin{rem}
We note that \cite[Lemma 2.19, Corollary 2.20, Theorem 2.21]{has-mmp-normal-pair} hold in the complex analytic setting. 
The proofs in the algebraic case work with no changes, thus we omit the proofs. 
\end{rem}

\end{sing}

\begin{quasi-log}\label{subsec--quasi-log}

In this subsection, we define a {\em quasi-log complex analytic space induced by a normal pair} (Definition \ref{defn--lc-trivial-fib-quasi-log}). 

\begin{defn}[{Quasi-log complex analytic space, \cite[Definition 1.1]{fujino-quasi-log-analytic}}]\label{defn--quasi-log}
A {\em{quasi-log complex analytic space}} is a complex analytic space $X$ endowed with a globally $\mathbb R$-Cartier divisor (or an $\mathbb R$-line bundle) $\omega$ on $X$, a closed complex analytic subspace ${\rm Nqlc}(X, \omega)\subsetneq X$ called the {\em non-qlc locus of $X$}, and a finite collection $\{C\}$ of reduced and irreducible closed complex analytic subspaces of $X$ such that there is a projective morphism $f\colon (Y, B_Y)\to X$ from an analytically globally embedded simple normal crossing pair satisfying the following properties: 
\begin{itemize}
\item $f^*\omega\sim_{\mathbb R}K_Y+B_Y$, 

\item the natural map $\mathcal O_X \to f_*\mathcal O_Y(\lceil -(B_Y^{<1})\rceil)$ induces an isomorphism 
$$
\mathcal I_{{\rm Nqlc}(X, \omega)}\overset{\cong}{\longrightarrow} f_*\mathcal O_Y(\lceil -(B_Y^{<1})\rceil-\lfloor B_Y^{>1}\rfloor),  
$$ 
where $\mathcal I_{{\rm Nqlc}(X, \omega)}$ is the defining ideal sheaf of ${\rm Nqlc}(X, \omega)$, and

\item the collection of closed analytic subvarieties
$\{C\}$ coincides with the images 
of the strata of $(Y, B_Y)$ that are not included in ${\rm Nqlc}(X, \omega)$. 
\end{itemize}

We simply write $[X, \omega]$ to denote the above data 
$$
\left(X, \omega, f\colon (Y, B_Y)\to X\right)
$$ 
if there is no risk of confusion. 
An element of $\{C\}$ is called a {\em qlc center of $[X,\omega]$}. 

The {\em non-qklt locus of $X$}, denoted by ${\rm Nqklt}(X, \omega)$, is the union of ${\rm Nqlc}(X, \omega)$ and all qlc centers of $[X,\omega]$. 
We note that ${\rm Nqklt}(X, \omega)$ has the complex analytic subspace structure naturally induced by ${\rm Nqlc}(X, \omega)$ and all qlc centers of $[X,\omega]$.  
\end{defn}

\begin{thm}[cf.~{\cite[Theorem 2.23]{has-mmp-normal-pair}}]\label{thm--abundance-quasi-log}
Let $\pi \colon X \to Z$ be a projective morphism to a Stein space $Z$, and let $W$ be a compact subset such that $\pi$ and $W$ satisfy (P). 
Let $[X,\omega]$ be a quasi-log complex analytic space. 
Let $A$ be a $\pi$-ample $\mathbb{R}$-divisor on $X$. 
Suppose that $\omega+A$ is nef over $W$ and $(\omega+A)|_{{\rm Nqlc}(X, \omega)}$, which we think of as an $\mathbb{R}$-line bundle on ${\rm Nqlc}(X, \omega)$, is semi-ample over a neighborhood of $W$. 
Then $\omega+A$ is semi-ample over a neighborhood of $W$. 
\end{thm}

\begin{proof} The argument from the proof of \cite[Theorem 2.23]{has-mmp-normal-pair} works with no changes because we may use the base point free theorem (\cite[Theorem 6.1]{fujino-quasi-log-analytic}) and the cone and contraction theorem for quasi-log complex analytic spaces (\cite[Theorem 9.2]{fujino-quasi-log-analytic}). \end{proof}

\begin{lem}[cf.~{\cite[Lemma 2.24]{has-mmp-normal-pair}}]\label{lem--str-quasi-log}
Let $f \colon Y \to X$ be a contraction of normal analytic varieties and let $(Y,\Delta)$ be a normal pair such that $f({\rm Nlc}(Y,\Delta))\subsetneq X$ and $K_{Y}+\Delta\sim_{\mathbb{R}}f^{*}\omega$ for some globally $\mathbb{R}$-Cartier divisor $\omega$ on $X$. 
Let $g \colon W \to Y$ be a log resolution of $(Y,\Delta)$. 
We define an $\mathbb{R}$-divisor $\Gamma$ on $W$ by 
$K_{W}+\Gamma=g^{*}(K_{Y}+\Delta).$ 
Let $(X,\omega, f\circ g \colon (W,\Gamma) \to X)$ be the structure of a quasi-log complex analytic space. 
Then the structure does not depend on the choice of $g \colon W \to Y$, that is, ${\rm Nqlc}(X, \omega)$ and the set of qlc centers do not depend on $g$.  
\end{lem}

\begin{proof}
The argument from the proof of \cite[Proposition 6.3.1]{fujino-book} works with no changes. 
\end{proof}

\begin{defn}[Quasi-log complex analytic space induced by normal pair]\label{defn--lc-trivial-fib-quasi-log}
A {\em quasi-log complex analytic space induced by a normal pair}, denoted by $f\colon (Y,\Delta) \to [X,\omega]$ in this paper, consists of a normal pair $(Y,\Delta)$, a contraction $f \colon Y \to X$ of normal analytic varieties, and the structure of a quasi-log complex analytic space $[X,\omega]$ on $X$ which is defined with a log resolution of $(Y,\Delta)$. 

By definition, we have
\begin{itemize}
\item
$f({\rm Nlc}(Y,\Delta))\subsetneq X$, and 
\item
$K_{Y}+\Delta\sim_{\mathbb{R}}f^{*}\omega$. 
\end{itemize}
By Lemma \ref{lem--str-quasi-log}, the structure of $[X,\omega]$ does not depend on the choice of a log resolution of $(Y,\Delta)$.

In the case of $X=Y$ and $\omega=K_{Y}+\Delta$, we may identify $[Y,K_{Y}+\Delta]$ with the normal pair $(Y,\Delta)$. 
In this situation, it follows that ${\rm Nqlc}(Y,K_{Y}+\Delta)={\rm Nlc}(Y,\Delta)$ and ${\rm Nqklt}(Y,K_{Y}+\Delta)={\rm Nklt}(Y,\Delta)$ as closed complex analytic spaces of $Y$. 
\end{defn}

\begin{rem}[cf.~{\cite[Remark 2.26]{has-mmp-normal-pair}}]\label{rem--quasi-log-non-qklt-image}
Let $f\colon (Y,\Delta) \to [X,\omega]$ be a quasi-log complex analytic space induced by a normal pair $(Y,\Delta)$. 
Then $f({\rm Nklt}(Y,\Delta))={\rm Nqklt}(X,\omega)$ and $f({\rm Nlc}(Y,\Delta))={\rm Nqlc}(X,\omega)$ set-theoretically.  
\end{rem}

\begin{lem}[cf.~{\cite[Lemma 2.28]{has-mmp-normal-pair}}]\label{lem--can-bundle-formula}
Let $f\colon (Y,\Delta) \to [X,\omega]$ be a quasi-log complex analytic space induced by a normal pair. 
Let $\pi \colon X \to Z$ be a projective morphism from $X$ to a Stein space $Z$, and let $W \subset Z$ be a Stein compact subset.  
Let $A$ be a $\pi$-ample $\mathbb{R}$-divisor on $X$. 
Then, after shrinking $Z$ around $W$, there exists a normal pair $(X,G)$ such that the relation $K_{X}+G \sim_{\mathbb{R},\,Z}\omega+A$ holds and ${\rm Nklt}(X,G)={\rm Nqklt}(X,\omega)$ as closed complex analytic subspaces of $X$.
\end{lem}

\begin{proof}
The argument from the proof of \cite[Lemma 2.28]{has-mmp-normal-pair} works without any changes because we can use \cite[Theorem 1.1]{has-gen-can-bundle-formula} instead of \cite[Corollary 5.2]{fujino-hashizume-adjunction}. 
\end{proof}

\begin{thm}[cf.~{\cite[Theorem 2.29]{has-mmp-normal-pair}}]\label{thm--vanishing-quasi-log}
Let $f\colon (Y,\Delta) \to [X,\omega]$ be a quasi-log complex analytic space induced by a normal pair. 
Fix $X'$ a union of ${\rm Nqlc}(X, \omega)$ and (possibly empty) some qlc centers of $[X,\omega]$. 
Let $\pi \colon X \to Z$ be a projective morphism to a complex analytic space $Z$, and let $L$ be a globally $\mathbb{Q}$-Cartier Weil divisor on $X$ such that 
\begin{itemize}
\item
$L$ is Cartier on a neighborhood of $X'$, 
\item
$f^{*}L$ is a Weil divisor on $Y$, and 
\item
$L-\omega$ is $\pi$-nef and $\pi$-log big. 
\end{itemize}
Then $R^{i}\pi_{*}(\mathcal{I}_{X'}\otimes_{\mathcal{O}_{X}} \mathcal{O}_{X}(L))=0$ for every $i>0$, where $\mathcal{I}_{X'}$ is the defining ideal sheaf of $X'$.
\end{thm}

\begin{proof}
The argument from the proof of \cite[Theorem 2.29]{has-mmp-normal-pair} works with no changes because we can use the torsion-free theorem (\cite[Theorem 3.5 (i)]{fujino-quasi-log-analytic}) and the vanishing theorem for analytic simple normal crossing pairs \cite[Theorem 3.8]{fujino-quasi-log-analytic}. 
\end{proof}

\end{quasi-log}

\section{Running minimal model program}\label{sec--running-mmp}

In this section, we define a minimal model and a good minimal model for normal pairs, and we study how to construct a sequence of an MMP for normal pairs whose non-nef locus is disjoint from the non-lc locus.

\subsection{Minimal model and minimal model program}

In this subsection, we define minimal models for normal pairs. 
For the definition of some models for lc pairs, see \cite[Definition 3.1]{eh-analytic-mmp}. 

\begin{defn}[Minimal model, cf.~{\cite[Definition 3.1]{has-mmp-normal-pair}}]\label{defn--minmodel}
Let $X \to Z$ be a projective morphism from a normal analytic variety $X$ to a complex analytic space $Z$, and let $(X,\Delta)$ be a normal pair. 
Let $W \subset Z$ be a subset. 
Let $(X',\Delta')$ be a normal pair with a projective morphism $X' \to Z$, and let $\phi \colon X \dashrightarrow X'$ be a bimeromorphic map over $Z$. 
We say that $(X',\Delta')$ is a {\em minimal model of $(X,\Delta)$ over $Z$ around $W$} if 
\begin{itemize}
\item
for any prime divisor $P$ on $X$, we have 
$$a(P,X,\Delta) \leq a(P,X',\Delta'),$$
and the strict inequality holds if $P$ is $\phi$-exceptional,   
\item
for any prime divisor $P'$ on $X'$, we have 
$${\rm coeff}_{P'}(\Delta')=-a(P',X,\Delta),$$  
and the inequality $a(P',X,\Delta)\leq -1$ holds if $P'$ is $\phi^{-1}$-exceptional, and
\item
$K_{X'}+\Delta'$ is nef over $W$. 
\end{itemize}
We say that a minimal model $(X',\Delta')$ of $(X,\Delta)$ over $Z$ around $W$ is a {\em good minimal model} if $K_{X'}+\Delta'$ is semi-ample over a neighborhood of $W$. 
A {\em $\mathbb{Q}$-factorial minimal model} (resp.~{\em a $\mathbb{Q}$-factorial good minimal model}) {\em of $(X,\Delta)$ over $Z$ around $W$} is a minimal model (resp.~a good minimal model) $(X',\Delta')$ of $(X,\Delta)$ over $Z$ around $W$ such that $X'$ is $\mathbb{Q}$-factorial over $W$. 
\end{defn}

\begin{rem}\label{rem--analytic-models} The complex analytic analogs of \cite[Lemmas 3.2, 3.3, Proposition 3.4]{has-mmp-normal-pair} hold since the proofs in the algebraic case work in our situation. \end{rem}

Next, we define a step of a minimal model program. 
For the log minimal model program for lc pairs in the complex analytic setting, see \cite[Subsection 3.2]{eh-analytic-mmp}. 
See also Remark \ref{rem--mmp-fullgeneral} for the difference between \cite[Definition 3.5]{eh-analytic-mmp} and the minimal model program used in this paper (Definition \ref{defn--mmp-fullgeneral} below).

\begin{defn}[Minimal model program, cf.~{\cite[Definition 3.5]{eh-analytic-mmp}}]\label{defn--mmp-fullgeneral} 
Let $\pi \colon X \to Z$ be a projective morphism from a normal analytic variety $X$ to a complex analytic space $Z$, and let $W \subset Z$ be a subset. 
Let $D$ be an $\mathbb{R}$-Cartier divisor on $X$. 

A {\em step of a $D$-MMP over $Z$ around $W$} is the diagram
 $$
\xymatrix@R=16pt{
X\ar@{-->}[rr]^-{\phi}\ar[dr]\ar[ddr]_-{\pi}&&X'\ar[dl]\ar[ddl]^-{\pi'}\\
&V\ar[d]\\
&Z
}
$$
consisting of normal analytic varieties $X$, $X'$, and $V$, which are projective over $Z$, such that
\begin{itemize}
\item
$X \to V$ is a bimeromorphic morphism and $X' \to V$ is a small bimeromorphic morphism, 
\item
$-D$ is ample over $V$, and 
\item
$\phi_{*}D$ is $\mathbb{R}$-Cartier and ample over $V$. 
\end{itemize}
Sometimes we call $X \to V$ a {\em $D$-negative extremal contraction}. 
Note that $W$ plays no role in this definition, although we use the phrase “around $W$” to emphasize that the modification is taken over a neighborhood of $W$. 

A {\em sequence of steps of a $D$-MMP over $Z$ around $W$} is a pair of sequences $\{Z_{i}\}_{i \geq 1}$ and $\{\phi_{i}\colon X_{i} \dashrightarrow X'_{i}\}_{i\geq 1}$, where each $Z_{i} \subset Z$ is an open subset and each $\phi_{i}$ is a bimeromorphic contraction of normal analytic varieties over $Z_{i}$, such that
\begin{itemize}
\item
$Z_{i}\supset W$ and $Z_{i}\supset Z_{i+1}$ for every $i\geq 1$, 
\item
$X_{1}=\pi^{-1}(Z_{1})$ and $X_{i+1}=X'_{i}\times_{Z_{i}}Z_{i+1}$ for every $i\geq 1$, and
\item
if we put $D_{1}:=D|_{X_{1}}$ and $D_{i+1}:=(\phi_{i*}D_{i})|_{X_{i+1}}$ for every $i\geq 1$, then 
$$X_{i} \dashrightarrow X'_{i}$$ 
is a step of a $D_{i}$-MMP over $Z_{i}$ around $W$.  
\end{itemize}
For the simplicity of notation, a sequence of steps of a $D$-MMP over $Z$ around $W$ is denoted by
$$X_{1} \dashrightarrow X_{2} \dashrightarrow \cdots \dashrightarrow X_{i}\dashrightarrow \cdots.$$
When $D$ is of the form $K_{X}+B$ for some normal pair $(X,B)$, then a sequence of steps of a $(K_{X}+B)$-MMP over $Z$ around $W$ is denoted by
$$(X_{1},B_{1})\dashrightarrow (X_{2},B_{2}) \dashrightarrow \cdots \dashrightarrow (X_{i},B_{i}) \dashrightarrow \cdots.$$

Let $H$ be an $\mathbb{R}$-Cartier divisor on $X$ such that $D+\lambda H$ is nef over $W$ for some $\lambda \in \mathbb{R}_{\geq 0}$. 
We put 
$$H_{1}:=H|_{X_{1}} \qquad {\rm and} \qquad H_{i+1}:=(\phi_{i*}H_{i})|_{X_{i+1}}$$
for each $i \geq 1$. 
Then a {\em sequence of steps of a $D$-MMP over $Z$ around $W$ with scaling of $H$} is a sequence of steps of a $D$-MMP over $Z$ around $W$
$$X=:X_{1}\dashrightarrow X_{2} \dashrightarrow \cdots \dashrightarrow X_{i} \dashrightarrow \cdots$$
with the data $\{Z_{i}\}_{i \geq 1}$ and $\{\phi_{i}\colon X_{i} \dashrightarrow X'_{i}\}_{i\geq 1}$ such that

\begin{itemize}
\item
$H_{i}$ is $\mathbb{R}$-Cartier, 
\item
the nonnegative real number
$$\lambda_{i}:={\rm inf}\{\mu \in \mathbb{R}_{\geq 0}\,|\, \text{$D_{i}+\mu H_{i}$ is nef over $W$}\}$$
is well defined, and 
\item
$(D_{i}+\lambda_{i}H_{i})\cdot C_{i}=0$ for any curve $C_{i} \subset X_{i}$ over $W$ that is contracted by the $D_{i}$-negative extremal contraction of the MMP. 
\end{itemize}
\end{defn}

\begin{rem}\label{rem--mmp-fullgeneral}
With notation as in Definition \ref{defn--mmp-fullgeneral}, the morphism $X \to V$ in a step of a $D$-MMP over $Z$ around $W$ (with scaling of $H$) does not necessarily satisfy $\rho(X/Z;W)-\rho(V/Z;W)=1$. 
In particular, $X \to V$ and $X' \to V$ in Definition \ref{defn--mmp-fullgeneral} can be biholomorphisms. 
\end{rem}

\begin{defn}[{\cite[Definition 3.6]{eh-analytic-mmp}}]\label{def--mmp-represent-bimerocont}
With notation as in Definition \ref{defn--mmp-fullgeneral}, let
$$X=:X_{1}\dashrightarrow X_{2} \dashrightarrow \cdots \dashrightarrow X_{i} \dashrightarrow \cdots$$
be a sequence of steps of a $D$-MMP over $Z$ around $W$ defined with the data $\{Z_{i}\}_{i \geq 1}$ and $\{\phi_{i}\colon X_{i} \dashrightarrow X'_{i}\}_{i\geq 1}$. 
We say that the $D$-MMP is {\em represented by bimeromorphic contractions} if $Z_{i}=Z$ for all $i \geq 1$. 
As in Remark \ref{rem--mmp-fullgeneral}, the morphism $X \to V$ in a step of a $D$-MMP over $Z$ around $W$ does not necessarily satisfy $\rho(X/Z;W)-\rho(V/Z;W)=1$. 

\end{defn}

Note that the complex analytic analog of \cite[Remark 3.7]{has-mmp-normal-pair} holds.

\begin{lem}[cf.~{\cite[Lemma 3.8]{has-mmp-normal-pair}}]\label{lem--extremal-ray}
Let $[X,\omega]$ be a quasi-log complex analytic space such that $X$ is a normal analytic variety. 
Let $\pi \colon X \to Z$ be a projective morphism to a Stein space $Z$, and let $W \subset Z$ be a compact subset such that $\rho (X/Z; W) < \infty$. 
Suppose that ${\rm NNef}(\omega/Z) \cap {\rm Nqlc}(X, \omega) \cap \pi^{-1}(W)=\emptyset$ (in particular, $\omega$ is $\pi$-pseudo-effective since ${\rm NNef}(\omega/Z)$ is well defined). 
Let $A$ be an effective globally $\mathbb{R}$-Cartier divisor on $X$ such that $\omega+A$ is nef over $W$ and ${\rm Nqlc}(X,\omega+A) = {\rm Nqlc}(X, \omega)$ set-theoretically. 
We put 
$$\lambda:={\rm inf}\{\mu \in \mathbb{R}_{\geq0}\,|\, \omega+\mu A \text{ is nef over $W$}\,\}.$$
Then $\lambda=0$ or there exists an $\omega$-negative extremal ray $R$ of $\overline{\rm NE}(X/Z;W)$ such that $R$ is rational and relatively ample at infinity (\cite[Definition 6.7.2]{fujino-book}) and $(\omega+\lambda A)\cdot R=0$. 
\end{lem}

\begin{proof}
The argument from the proof of \cite[Lemma 3.8]{has-mmp-normal-pair} works with no changes because we can use the cone and contraction theorem for quasi-log complex analytic spaces and the length of extremal curves (\cite[Section 9]{fujino-quasi-log-analytic}). 
\end{proof}

\subsection{Construction of minimal model program}
The goal of this subsection is to prove the complex analytic analogs of \cite[Corollaries 3.12, 3.13]{has-mmp-normal-pair}. 
To prove them, we need the complex analytic analog of \cite[Theorem 3.9]{has-mmp-normal-pair}. 
However, the proof is not verbatim because the proof of \cite[Theorem 3.9]{has-mmp-normal-pair} needs the gluing of MMP, which is not easy in the complex analytic setting. 
Thus, we need to prepare some lemmas for the proof. 

\begin{lem}[cf.~{\cite[Lemma 3.2]{has-flop}}]\label{lem--rho-decrease}
Let $X \to Z$ be a projective morphism from a normal analytic variety $X$ to a complex analytic space $Z$, and let $W \subset Z$ be a compact subset. 
Suppose that $X$ is $\mathbb{Q}$-factorial over $W$. 
Let $X \dashrightarrow X'$ be a bimeromorphic contraction over $Z$ to a normal complex variety $X'$. 
Then $\rho(X'/Z;W) \leq \rho(X/Z;W)$. 
\end{lem}

\begin{proof}
The argument from the proof of \cite[Lemma 3.2]{has-flop} works with no changes. 
Note that we do not need to assume $\rho(X/Z;W)<\infty$. 
\end{proof}

\begin{lem}\label{lem--bimerocont-small}
Let $X \to Z$ be a projective morphism from a normal analytic variety to a complex analytic space, and let $W \subset Z$ be a compact subset such that $\rho(X/Z;W)<\infty$.  
Suppose that we are given the data $\{Z_{i}\}_{i \geq 1}$ and $\{\phi_{i}\colon X_{i} \dashrightarrow X'_{i}\}_{i\geq 1}$, where $Z_{i} \subset Z$ are open subsets and $\phi_{i}$ are bimeromorphic contractions of normal analytic varieties over $Z_{i}$, such that
\begin{itemize}
\item
$Z_{i}\supset W$ and $Z_{i}\supset Z_{i+1}$ for every $i\geq 1$, and
\item
$X_{1}=\pi^{-1}(Z_{1})$ and $X_{i+1}=X'_{i}\times_{Z_{i}}Z_{i+1}$ for every $i\geq 1$. 
\end{itemize}
 Then there exists an index $i_{0}$ such that $\phi_{i}$ does not have any exceptional divisor whose image on $Z_{i}$ intersects $W$ for any $i \geq i_{0}$.  
In particular, $X_{i} \dashrightarrow X'_{i}$ is small over an open neighborhood $U_{i}$ of $W$ for any $i \geq i_{0}$. 
\end{lem}

\begin{proof}
By replacing $X$ with a resolution, we may assume that $X$ is non-singular. 
We set $d :=\rho(X/Z;W)$. 
Suppose by contradiction that there exist infinitely many indices $i$ such that $X_{i} \dashrightarrow X'_{i}$ has an exceptional divisor whose image on $Z_{i}$ intersects $W$. 
Then we can find an index $j$ such that after replacing $Z$ with $Z_{j}$, the bimeromorphic contraction $X \dashrightarrow X_{j}$ has distinct exceptional prime divisors $E_{1},\,\cdots,\,E_{d+1}$ whose images on $Z$ intersect $W$. 
Since $d =\rho(X/Z;W)$, we can find real numbers $r_{1},\,\cdots,\,r_{d+1}$ such that at least one of them is not zero and $\sum_{k=1}^{d+1}r_{k}E_{k}$ is numerically trivial over $W$. 
We set $W_{j}$ as the inverse image of $W$ to $X_{j}$, and let $f \colon \tilde{X} \to X$ be a resolution of the indeterminacy of $X \dashrightarrow X_{j}$. 
Then all $f^{*}E_{k}$ are exceptional over $X_{j}$ and $\sum_{k=1}^{d+1}r_{k}f^{*}E_{k} $ is numerically trivial over $W_{j}$.  
By applying the negativity lemma \cite[Corollary 2.16]{eh-analytic-mmp} to $\tilde{X} \to X_{j}$ and $\sum_{k=1}^{d+1}r_{k}f^{*}E_{k}$, after shrinking $Z$ around $W$ we have $\sum_{k=1}^{d+1}r_{k}f^{*}E_{k} = 0$ as an $\mathbb{R}$-divisor. 
Then $\sum_{k=1}^{d+1}r_{k}E_{k} = 0$ as an $\mathbb{R}$-divisor, from which we get a contradiction because $E_{1},\,\cdots,\,E_{d+1}$ are distinct prime divisors. 
Thus, Lemma \ref{lem--bimerocont-small} holds. 
\end{proof}

The following result is Lemma \ref{lem--bimerocont-small} in the algebraic setting. 
Although we do not use this result in this paper, we include its precise statement because it may be useful and does not seem to appear in the literature. 
The proof is verbatim. 

\begin{lem}
Let $X \to Z$ be a projective morphism of normal quasi-projective algebraic varieties. 
Let $X \dashrightarrow X_{1}\dashrightarrow \cdots \dashrightarrow X_{i}\dashrightarrow \cdots$
be a sequence of birational contractions of algebraic varieties, where $X_{i}$ are normal and projective over $Z$. 
Then there exists an index $i_{0}$ such that $X_{i} \dashrightarrow X_{i+1}$ is small for every $i \geq i_{0}$.
\end{lem}

The following lemma is very similar to \cite[Theorem 3.15]{eh-analytic-mmp}. 
In view of Remark \ref{rem--mmp-fullgeneral}, the only difference between Lemma \ref{lem--termination-mmp-generalized} below and \cite[Theorem 3.15]{eh-analytic-mmp} is the relative Picard number of the extremal contraction in each step of the relative MMP with scaling. 
However, the difference is harmless, and thus we may use the argument from the proof of \cite[Theorem 3.15]{eh-analytic-mmp} directly. 
Hence, we skip the proof.

\begin{lem}[cf.~{\cite[Theorem 3.15]{eh-analytic-mmp}}]\label{lem--termination-mmp-generalized}
Let $\pi \colon X \to Z$ be a projective morphism from a normal analytic variety $X$ to a Stein space $Z$, and let $W \subset Z$ be a compact subset such that $\pi$ and $W$ satisfy (P). 
Let $(X,\Delta)$ be an lc pair. 
Let $H$ be an effective $\mathbb{R}$-Cartier divisor on $X$ such that $(X,\Delta+H)$ is an lc pair, and let
$$(X_{1},\Delta_{1})\dashrightarrow (X_{2},\Delta_{2})\dashrightarrow \cdots \dashrightarrow (X_{i},\Delta_{i})\dashrightarrow \cdots$$
be a sequence of steps of a $(K_{X}+\Delta)$-MMP over $Z$ around $W$ with scaling of $H$ in the sense of Definition \ref{defn--mmp-fullgeneral}. 
For each $i \geq 1$, put
$$\lambda_{i}:={\rm inf}\{\mu \in \mathbb{R}_{\geq 0}\,|\, \text{$K_{X_{i}}+\Delta_{i}+\mu H_{i}$ {\rm is nef over} $W$}\}$$
If we have ${\rm lim}_{i \to \infty}\lambda_{i}=0$ and $(X,\Delta)$ has a weak lc model over $Z$ around $W$ {\rm (}{\rm cf.}~\cite[Definition 3.1]{eh-analytic-mmp}{\rm )}, then we have $\lambda_{n}=0$ for some $n \in \mathbb{Z}_{\geq 0}$, that is, the given MMP with scaling terminates. 
\end{lem}

The following technical lemma is crucial for the proof of \cite[Theorem 3.9]{has-mmp-normal-pair} in the complex analytic setting.

\begin{lem}\label{lem--mmp-normal-pair-run}
Let $f \colon X \to \widetilde{V}$ be a projective morphism of normal analytic varieties, and let $\pi \colon \widetilde{V} \to Z$ be a projective morphism to a Stein space $Z$. 
Let $(X,\Delta)$ be a normal pair. 
Let $W \subset Z$ be a compact subset such that $\pi \circ f \colon X \to Z$ and $W$ satisfy (P). 
Set $\widetilde{W}:=\pi^{-1}(W)\subset \widetilde{V}$. 
Suppose that the following conditions hold:
\begin{itemize}
\item
There exists an open subset $\widetilde{U} \subset \widetilde{V}$, which may not contain $\widetilde{W}$, such that ${\rm Nlc}(X,\Delta) \subset f^{-1}(\widetilde{U})$ and $(K_{X}+\Delta)|_{f^{-1}(\widetilde{U})}$ is semi-ample over $\widetilde{U}$, and
\item 
for each point $w \in \widetilde{W}\setminus \widetilde{U}$, there exists an open subset $\widetilde{U}_{w} \ni w$ of $\widetilde{V}$, which may not contain $\widetilde{W}$, such that $(f^{-1}(\widetilde{U}_{w}), \Delta|_{f^{-1}(\widetilde{U}_{w})})$ has a good minimal model over $\widetilde{U}_{w}$ around $w$. 
\end{itemize}
We set $(X_{1},\Delta_{1})=(X,\Delta)$. 
Then, after shrinking $Z$ around $W$, there exists a sequence of  steps of a $(K_{X_{1}}+\Delta_{1})$-MMP over $\widetilde{V}$ around $\widetilde{W}$
$$
\xymatrix{
(X_{1},\Delta_{1})\ar@{-->}[r]&\cdots \ar@{-->}[r]& (X_{i},\Delta_{i})\ar@{-->}[rr]\ar[dr]&& (X_{i+1},\Delta_{i+1})\ar[dl] \ar@{-->}[r]&\cdots\ar@{-->}[r]&(X_{m},\Delta_{m}), \\
&&&V_{i}
}
$$
which is represented by bimeromorphic contractions, to a good minimal model $(X_{m},\Delta_{m})$ over $\widetilde{V}$ around $\widetilde{W}$ such that 
\begin{itemize}
\item
the non-biholomorphic locus of the MMP is disjoint from ${\rm Nlc}(X,\Delta)$, and
\item
$\rho(X_{i}/\widetilde{V};\widetilde{W})-\rho(V_{i}/\widetilde{V};\widetilde{W})=1$ for every $i \geq 1$.
\end{itemize}
\end{lem}

\begin{proof}
We divide the proof into several steps. 

\begin{step5}\label{step1--lem--mmp-normal-pair-run}
Let $A_{1}$ be an arbitrary $(\pi \circ f)$-ample $\mathbb{R}$-Cartier divisor on $X_{1}$. 
In this step, we prove that we can construct a sequence of steps of a $(K_{X_{1}}+\Delta_{1})$-MMP over $\widetilde{V}$ around $\widetilde{W}$ with scaling of $A_{1}$ 
$$
\xymatrix{
(X_{1},\Delta_{1})\ar@{-->}[r]&\cdots \ar@{-->}[r]& (X_{i},\Delta_{i})\ar@{-->}[rr]\ar[dr]&& (X_{i+1},\Delta_{i+1})\ar[dl] \ar@{-->}[r]&\cdots\ar@{-->}[r]&\cdots \\
&&&V_{i}
}
$$
such that 
\begin{itemize}
\item
the non-biholomorphic locus of the MMP is disjoint from ${\rm Nlc}(X,\Delta)$, and
\item
$\rho(X_{i}/\widetilde{V};\widetilde{W})-\rho(V_{i}/\widetilde{V};\widetilde{W})=1$ for every $i \geq 1$.
\end{itemize}

Shrinking $Z$ around $W$, we may assume that $A_{1}$ is globally $\mathbb{R}$-Cartier. 
Rescaling $A_{1}$ and choosing $A_{1}$ as a general member of $|A_{1}/Z|_{\mathbb{R}}$, we may assume that $K_{X_{1}}+\Delta_{1}+A_{1}$ is nef over $\widetilde{W}$ and ${\rm Nlc}(X_{1},\Delta_{1})={\rm Nlc}(X_{1},\Delta_{1}+A_{1})$ set-theoretically. 
Set
$$\lambda_{1}:={\rm inf}\{\mu \in \mathbb{R}_{\geq 0}\,|\, \text{$K_{X_{1}}+\Delta_{1}+\mu A_{1}$ {\rm is nef over} $\widetilde{W}$}\}$$
We have $\rho(X/\widetilde{V};\widetilde{W}) \leq \rho(X/Z;W)$ by the standard argument. 
Thus $\rho(X/\widetilde{V};\widetilde{W}) < \infty$. 
By Lemma \ref{lem--extremal-ray} and the cone and contraction theorem \cite[Theorem 9.2]{fujino-quasi-log-analytic}, shrinking $Z$ around $W$, we get a contraction $\varphi\colon X_{1} \to V_{1}$ over $\widetilde{V}$ that contracts a $(K_{X_{1}}+\Delta_{1})$-negative extremal ray $R$ such that $(K_{X_{1}}+\Delta_{1}+\lambda_{1}A_{1})\cdot R =0$. 
Since $(K_{X}+\Delta)|_{f^{-1}(\widetilde{U})}$ is semi-ample over $\widetilde{U}$ by assumption, we see that $\varphi$ is a biholomorphism on $f^{-1}(\widetilde{U})$. 
Therefore, $\varphi$ is bimeromorphic. 
Let $V'_{1} \subset V_{1}$ be the largest open subset over which $\varphi$ is a biholomorphism. 
We put $V''_{1}=V_{1} \setminus \varphi({\rm Nlc}(X,\Delta))$, and we put
$$(X'_{1},\Delta'_{1}):=(\varphi^{-1}(V'_{1}), \Delta_{1}|_{\varphi^{-1}(V'_{1})}), \quad {\rm and} \quad (X''_{1},\Delta''_{1}):=(\varphi^{-1}(V''_{1}), \Delta_{1}|_{\varphi^{-1}(V''_{1})}).$$ 
Note that $V'_{1} \supset \varphi({\rm Nlc}(X,\Delta))$ since ${\rm Nlc}(X,\Delta) \subset f^{-1}(\widetilde{U})$ by assumption and $\varphi$ is a biholomorphism on $f^{-1}(\widetilde{U})$. 
By \cite[Theorem 1.7]{fujino-analytic-lcabundance}, there exists an lc pair $(\overline{X}_{1},\overline{\Delta}_{1})$ with a projective small bimeromorphic morphism $\overline{X}_{1} \to V''_{1}$ such that $\overline{\Delta}_{1}$ is the strict transform of $\Delta''_{1}$ on $\overline{X}_{1}$ and $K_{\overline{X}_{1}}+\overline{\Delta}_{1}$ is ample over $V''_{1}$. 
By using Remmert-Stein's theorem \cite[Theorem~4.6]{Shi}, we may glue the diagrams
$$
\xymatrix{
(X''_{1},\Delta''_{1}) \ar[dr]\ar@{-->}[rr]&&(\overline{X}_{1},\overline{\Delta}_{1}) \ar[dl]\\
&V''_{1}
}
\qquad{\rm and}\qquad
\xymatrix{
(X'_{1},\Delta'_{1}) \ar[dr]\ar[rr]^{\rm id}&&(X'_{1},\Delta'_{1}) \ar[dl]\\
&V'_{1}
}
$$
over the open subset $V'_{1} \cap V''_{1}$ and we get a step of a $(K_{X_{1}}+\Delta_{1})$-MMP over $\widetilde{V}$ around $\widetilde{W}$ with scaling of $A_{1}$
$$
\xymatrix{
(X_{1},\Delta_{1}) \ar[dr]\ar@{-->}[rr]&&(X_{2},\Delta_{2}) \ar[dl]\\
&V_{1}
}
$$
such that the bimeromorphic contraction $X_{1} \dashrightarrow X_{2}$ is a biholomorphism on ${\rm Nlc}(X_{1},\Delta_{1})$ and $\rho(X_{1}/\widetilde{V};\widetilde{W})-\rho(V_{1}/\widetilde{V};\widetilde{W})=1$. 
  
Repeating this discussion taking also \cite[Theorem 9.2 (4) (iii)]{fujino-quasi-log-analytic} into account, we can construct a sequence of steps of a $(K_{X_{1}}+\Delta_{1})$-MMP over $\widetilde{V}$ around $\widetilde{W}$ with scaling of $A_{1}$, as discussed at the start of this step. 
\end{step5} 

We note that for any sequence of steps of a $(K_{X_{1}}+\Delta_{1})$-MMP over $\widetilde{V}$ around $\widetilde{W}$, the bimeromorphic contraction $X_{i} \dashrightarrow X_{i+1}$ induces a linear map 
$$N^{1}(X_{i}/\widetilde{V};\widetilde{W}) \longrightarrow N^{1}(X_{i+1}/\widetilde{V};\widetilde{W})$$
 (see \cite[Theorem 3.9]{eh-analytic-mmp}, see also \cite[Remark 6.1 (2)]{hashizumehu}).

\begin{step5}
In this step, we prove that after replacing $(X_{1},\Delta_{1})$ with another normal pair, we may assume that any sequence of steps of a $(K_{X_{1}}+\Delta_{1})$-MMP over $\widetilde{V}$ around $\widetilde{W}$ 
$$(X_{1},\Delta_{1})\dashrightarrow (X_{2},\Delta_{2})\dashrightarrow \cdots \dashrightarrow (X_{i},\Delta_{i})\dashrightarrow \cdots$$
satisfies the condition that the induced linear map $N^{1}(X_{i}/\widetilde{V};\widetilde{W}) \longrightarrow N^{1}(X_{i+1}/\widetilde{V};\widetilde{W})$ is bijective for any $i \geq 1$. 

If there is a finite sequence of steps of a $(K_{X_{1}}+\Delta_{1})$-MMP over $\widetilde{V}$ around $\widetilde{W}$  
$$(X_{1},\Delta_{1})\dashrightarrow (X'_{1},\Delta'_{1})$$
which contracts a divisor, then we replace $(X_{1},\Delta_{1})$ by $(X'_{1},\Delta'_{1})$. 
By Lemma \ref{lem--bimerocont-small} and taking this replacement finitely many times, we may assume that any sequence of steps of a $(K_{X_{1}}+\Delta_{1})$-MMP over $\widetilde{V}$ around $\widetilde{W}$ 
$$(X_{1},\Delta_{1})\dashrightarrow (X_{2},\Delta_{2})\dashrightarrow \cdots \dashrightarrow (X_{i},\Delta_{i})\dashrightarrow \cdots$$
does not contract any divisor. 
Then the linear map $N^{1}(X_{i}/\widetilde{V};\widetilde{W}) \to N^{1}(X_{i+1}/\widetilde{V};\widetilde{W})$ is injective for any $i \geq 1$. 
In particular, $\rho(X_{i}/\widetilde{V};\widetilde{W}) \leq \rho(X_{i+1}/\widetilde{V};\widetilde{W})$. 

We pick a projective morphism $Y \to \widetilde{V}$ from a non-singular analytic variety $Y$ with a bimeromorphic contraction $Y \dashrightarrow X$, for instance, we pick a resolution $Y$ of $X$. 
If there is a finite sequence of steps of a $(K_{X_{1}}+\Delta_{1})$-MMP over $\widetilde{V}$ around $\widetilde{W}$  
$$(X_{1},\Delta_{1})\dashrightarrow (X''_{1},\Delta''_{1})$$
such that $\rho(X_{1}/\widetilde{V};\widetilde{W}) < \rho(X''_{1}/\widetilde{V};\widetilde{W})$, then we replace $(X_{1},\Delta_{1})$ with $(X''_{1},\Delta''_{1})$. 
By Lemma \ref{lem--rho-decrease}, the relation $\rho(X/\widetilde{V};\widetilde{W}) \leq \rho(Y/\widetilde{V};\widetilde{W})<\infty$ always holds, even after the replacement. 
Hence, taking this replacement finitely many times, we may assume that any sequence of steps of a $(K_{X_{1}}+\Delta_{1})$-MMP over $\widetilde{V}$ around $\widetilde{W}$ 
$$(X_{1},\Delta_{1})\dashrightarrow (X_{2},\Delta_{2})\dashrightarrow \cdots \dashrightarrow (X_{i},\Delta_{i})\dashrightarrow \cdots$$
satisfies the condition that $\rho(X_{i}/\widetilde{V};\widetilde{W}) = \rho(X_{i+1}/\widetilde{V};\widetilde{W})$ for any $i \geq 1$. 
Then 
$$N^{1}(X_{i}/\widetilde{V};\widetilde{W}) \longrightarrow N^{1}(X_{i+1}/\widetilde{V};\widetilde{W})$$
is bijective. 
We finish this step. 
\end{step5}

\begin{step5}
With this step we finish the proof.

By the same argument as in the proof of \cite[Lemma 2.7]{eh-analytic-mmp-2} (see also the proof of \cite[Proposition 6.2]{hashizumehu}), after shrinking $Z$ around $W$, there is a $(\pi \circ f)$-ample $\mathbb{R}$-Cartier divisor $H_{1}$ on $X_{1}$ and a sequence of steps of a $(K_{X_{1}}+\Delta_{1})$-MMP over $\widetilde{V}$ around $\widetilde{W}$ with scaling of $H_{1}$ 
$$
\xymatrix{
(X_{1},\Delta_{1})\ar@{-->}[r]&\cdots \ar@{-->}[r]& (X_{i},\Delta_{i})\ar@{-->}[rr]\ar[dr]&& (X_{i+1},\Delta_{i+1})\ar[dl] \ar@{-->}[r]&\cdots\ar@{-->}[r]&\cdots \\
&&&V_{i}
}
$$
such that 
\begin{itemize}
\item
the non-biholomorphic locus of the MMP is disjoint from ${\rm Nlc}(X,\Delta)$, and
\item
$\rho(X_{i}/\widetilde{V};\widetilde{W})-\rho(V_{i}/\widetilde{V};\widetilde{W})=1$ for every $i \geq 1$, and
\item
if we define
$$\lambda_{i}:={\rm inf}\{\mu \in \mathbb{R}_{\geq 0}\,|\, \text{$K_{X_{i}}+\Delta_{i}+\mu H_{i}$ {\rm is nef over} $\widetilde{W}$}\}$$
for each $i \geq 1$ and $\lambda:={\rm lim}_{i \to \infty} \lambda_{i}$, then $\lambda\neq \lambda_{i}$ for all $i$. 
\end{itemize}
We put $H:=H_{1}$, which is a $(\pi \circ f)$-ample $\mathbb{R}$-Cartier divisor on $X=X_{1}$. 

Fix a point $w \in \widetilde{W}$. 
In this paragraph, we find an index $i_{w}$ and an open subset $\widetilde{U}_{w}\subset \widetilde{V}$ containing $w$ such that $K_{X_{i_{w}}}+\Delta_{i_{w}}+\lambda H_{i_{w}}$ is semi-ample over $\widetilde{U}_{w}$. 
If $w \in \widetilde{U}$, where $\widetilde{U}$ is the open subset in Lemma \ref{lem--mmp-normal-pair-run}, then we may set $i_{w}:=1$ and $\widetilde{U}_{w}:=\widetilde{U}$. 
Thus, we may assume $w \in \widetilde{W} \setminus \widetilde{U}$. 
By our assumption in Lemma \ref{lem--mmp-normal-pair-run}, there is an open subset $\widetilde{U}_{w} \ni w$ of $\widetilde{V}$ such that $(f^{-1}(\widetilde{U}_{w}), \Delta|_{f^{-1}(\widetilde{U}_{w})})$ has a good minimal model over $\widetilde{U}_{w}$ around $w$. 
Since ${\rm Nlc}(X,\Delta) \subset f^{-1}(\widetilde{U})$, which is our assumption of Lemma \ref{lem--mmp-normal-pair-run}, we may assume that $(f^{-1}(\widetilde{U}_{w}), \Delta|_{f^{-1}(\widetilde{U}_{w})})$ is lc. 
By \cite[Theorem 1.1]{eh-analytic-mmp}, \cite[Theorem 1.5]{fujino-analytic-lcabundance}, and shrinking $\widetilde{U}_{w}$ around $w$, we may assume that $\widetilde{U}_{w}$ is Stein and $(f^{-1}(\widetilde{U}_{w}), (\Delta+\lambda H)|_{f^{-1}(\widetilde{U}_{w})})$ has a good minimal model over $\widetilde{U}_{w}$ around $w$ even if $\lambda>0$. 
Now we take the restriction of the above $(K_{X_{1}}+\Delta_{1})$-MMP
$$(X_{1},\Delta_{1})\dashrightarrow (X_{2},\Delta_{2})\dashrightarrow \cdots \dashrightarrow (X_{i},\Delta_{i})\dashrightarrow \cdots$$
over $\widetilde{U}_{w}$, and we get a sequence of steps of a $(K_{f^{-1}(\widetilde{U}_{w})}+(\Delta+\lambda H)|_{f^{-1}(\widetilde{U}_{w})})$-MMP over $\widetilde{U}_{w}$ around $w$ with scaling of $H|_{f^{-1}(\widetilde{U}_{w})}$. 
By Lemma \ref{lem--termination-mmp-generalized} and shrinking $\widetilde{U}_{w}$ if necessary, we can find an index $i_{w}$  such that $K_{X_{i_{w}}}+\Delta_{i_{w}}+\lambda H_{i_{w}}$ is semi-ample over $\widetilde{U}_{w}$. 

We consider the open covering $\bigcup_{w \in \widetilde{W}}\widetilde{U}_{w}$ of $\widetilde{W}$. 
Since $\widetilde{W}$ is compact, there are only finitely many points $w_{1},\,\cdots ,\,w_{p}$ such that $\widetilde{W} \subset \bigcup_{j=1}^{p}\widetilde{U}_{w_{j}}$. 
We set $m:=\underset{1\leq j \leq p}{\rm max}i_{w_{j}}$ and $\widetilde{U}:=\bigcup_{j=1}^{p}\widetilde{U}_{w_{j}}$. 
Then $K_{X_{m}}+\Delta_{m}+\lambda H_{m}$ is semi-ample over $\widetilde{U}$. 
This shows $\lambda\geq \lambda_{m}$. 
Therefore, $\lambda = \lambda_{m}=0$. 
Then the $(K_{X_{1}}+\Delta_{1})$-MMP over $\widetilde{V}$ around $\widetilde{W}$
$$(X_{1},\Delta_{1})\dashrightarrow (X_{2},\Delta_{2})\dashrightarrow \cdots \dashrightarrow (X_{m},\Delta_{m})$$
is the desired one. 
\end{step5}
We complete the proof.
\end{proof}

We are ready to prove the complex analytic analog of \cite[Theorem 3.9]{has-mmp-normal-pair}. 

\begin{thm}[cf.~{\cite[Theorem 3.9]{has-mmp-normal-pair}}]\label{thm--mmpstep-quasi-log}
Let $f\colon (Y,\Delta) \to [X,\omega]$ be a quasi-log complex analytic space induced by a normal pair. 
Let $\pi \colon X \to Z$ be a projective morphism to a Stein space $Z$ and $W \subset Z$ a Stein compact subset such that $\pi \circ f$ and $W$ satisfy (P). 
Let $\varphi\colon X \to V$ be a bimeromorphic morphism over $Z$, where $V$ is normal and projective over $Z$, such that $-\omega$ is $\varphi$-ample and $\varphi$ is a biholomorphism on a neighborhood of ${\rm Nqlc}(X, \omega)$. 
Then, after shrinking $Z$ around $W$, we can construct a diagram
$$
\xymatrix@R=16pt{
(Y,\Delta)\ar[d]_{f}\ar@{-->}[rr]&& (Y',\Delta')\ar[d]^{f'} \\
[X,\omega] \ar@{-->}[rr]\ar[dr]_{\varphi}&& [X',\omega'] \ar[dl]^{\varphi'}\\
&V
}
$$
over $Z$ such that
\begin{itemize}
\item
$f'\colon (Y',\Delta') \to [X',\omega']$ is a quasi-log complex analytic space induced by a normal pair such that $Y'$ and $X'$ are projective over $V$, 
\item
$(Y,\Delta)\dashrightarrow (Y',\Delta')$ is a sequence of steps of a $(K_{Y}+\Delta)$-MMP over $Z$ around $W$, and 
\item
$\varphi'\colon X' \to V$ is a projective small bimeromorphic morphism and the $\mathbb{R}$-divisor $\omega'$ is the strict transform of $\omega$ on $X'$ and $\varphi'$-ample. 
\end{itemize}
Furthermore, if $Y$ is $\mathbb{Q}$-factorial over $W$, then $Y'$ is also $\mathbb{Q}$-factorial over $W$. 
\end{thm}

\begin{proof}
Let $W'$ be the inverse image of $W$ to $V$. 
We check that $\varphi \circ f \colon (Y,\Delta) \to V$ and $W' \subset V$ satisfy the conditions of Lemma \ref{lem--mmp-normal-pair-run} after shrinking $Z$ around $W$ suitably. 
Let $U \subset V$ be the largest open subset over which $\varphi$ is a biholomorphism. 
Since $\varphi$ is a biholomorphism on a neighborhood of ${\rm Nqlc}(X, \omega)$, by Remark \ref{rem--quasi-log-non-qklt-image} we have ${\rm Nlc}(Y,\Delta) \subset (\varphi \circ f)^{-1}(U)$. 
Moreover, the property $K_{Y}+\Delta \sim_{\mathbb{R},\,X} 0$ implies that $(K_{Y}+\Delta)|_{(\varphi \circ f)^{-1}(U)}\sim_{\mathbb{R},\,U}0$. 
In particular, after shrinking $Z$ around $W$, the divisor $K_{Y}+\Delta$ is semi-ample over $U$. 
Thus, the first condition of Lemma \ref{lem--mmp-normal-pair-run} holds. 
To check the second condition of Lemma \ref{lem--mmp-normal-pair-run}, set $U':= V \setminus \varphi({\rm Nqlc}(X, \omega))$. 
By the $\varphi$-ampleness of $-\omega$, there is a $\pi$-ample $\mathbb{R}$-divisor $H$ on $X$ such that $\omega+H \sim_{\mathbb{R},\,V}0$. 
Then $K_{Y}+\Delta + f^{*}H \sim_{\mathbb{R},\,V} 0$. 
By shrinking $Z$ around $W$ and replacing $H$ with a general member of $|H/Z|_{\mathbb{R}}$, we may assume that $(Y ,\Delta+f^{*}H)$ is lc over $U'$ and we have ${\rm Nlc}(Y,\Delta)={\rm Nlc}(Y,\Delta+f^{*}H)$. 
By \cite[Theorem 1.2]{eh-analytic-mmp} and \cite[Theorem 1.5]{fujino-analytic-lcabundance}, for each point $w \in U'$, there exists an open subset $U_{w} \ni w$ of $V$  such that $\bigl((\varphi \circ f)^{-1}(U_{w}), \Delta|_{(\varphi \circ f)^{-1}(U_{w})}\bigr)$ has a good minimal model over $U_{w}$ around $w$. 
Since $W' \setminus U \subset U'$,  the second condition of Lemma \ref{lem--mmp-normal-pair-run} holds. 
Thus, $\varphi \circ f \colon (Y,\Delta) \to V$ and $W'$ satisfy the conditions of Lemma \ref{lem--mmp-normal-pair-run} after shrinking $Z$ around $W$. 

We set $Y_{1}:=Y$ and $\Delta_{1}:=\Delta$. 
By Lemma \ref{lem--mmp-normal-pair-run}, after shrinking $Z$ around $W$, there exists a sequence of  steps of a $(K_{Y}+\Delta)$-MMP over $V$ around $W'$
$$
\xymatrix{
(Y_{1},\Delta_{1})\ar@{-->}[r]&\cdots \ar@{-->}[r]& (Y_{i},\Delta_{i})\ar@{-->}[rr]\ar[dr]&& (Y_{i+1},\Delta_{i+1})\ar[dl] \ar@{-->}[r]&\cdots\ar@{-->}[r]&(Y_{m},\Delta_{m}), \\
&&&V_{i}
}
$$
which is represented by bimeromorphic contractions, to a good minimal model $(Y_{m},\Delta_{m})$ over $V$ around $W'$ such that 
\begin{itemize}
\item
the non-biholomorphic locus of the MMP is disjoint from ${\rm Nlc}(Y,\Delta)$, and
\item
$\rho(Y_{i}/V;W')-\rho(V_{i}/V;W')=1$ for every $i \geq 1$.
\end{itemize}
Set $(Y',\Delta'):=(Y_{m},\Delta_{m})$. 
By shrinking $Z$ around $W$, we get a contraction $f' \colon Y' \to X'$ over $V$ induced by $K_{Y'}+\Delta'$. 
Then the induced morphism $\varphi' \colon X' \to V$ is a biholomorphism over $U$. 
Hence, $\varphi'$ is a bimeromorphic morphism. 
Let $\omega'$ be the strict transform of $\omega$ on $X'$. 
By the same argument as in the proof of \cite[Theorem 3.9]{has-mmp-normal-pair}, the diagram 
$$
\xymatrix@R=16pt{
(Y,\Delta)\ar[d]_{f}\ar@{-->}[rr]&& (Y',\Delta')\ar[d]^{f'} \\
[X,\omega] \ar@{-->}[rr]\ar[dr]_{\varphi}&& X' \ar[dl]^{\varphi'}\\
&V
}
$$
over $Z$ and $\omega'$ satisfy the conditions of Theorem \ref{thm--mmpstep-quasi-log}. 
This completes the proof. 
\end{proof}

\begin{thm}[cf.~{\cite[Theorem 3.11]{has-mmp-normal-pair}}]\label{thm--mmp-nefthreshold-strict}
Let $f\colon (Y,\Delta) \to [X,\omega]$ be a quasi-log complex analytic space induced by a normal pair. 
Let $\pi \colon X \to Z$ be a projective morphism to a Stein space $Z$, and let $W \subset Z$ be a compact subset such that $\pi \circ f \colon Y \to Z$ and $W$ satisfy (P). 
Suppose that $\omega$ is $\pi$-pseudo-effective and ${\rm NNef}(\omega/Z) \cap {\rm Nqlc}(X, \omega) \cap \pi^{-1}(W)=\emptyset$. 
Let $A$ be an $\mathbb{R}$-Cartier divisor on $X$ such that $\omega+\lambda_{0} A$ is ample over a neighborhood of $W$ for some positive real number $\lambda_{0}$. 
Then, after shrinking $Z$ around $W$, there exists a diagram
$$
\xymatrix{
(Y,\Delta)=:(Y_{k_{1}},\Delta_{k_{1}})\ar@<4.5ex>[d]_{f=:f_{1}}\ar@{-->}[r]& (Y_{k_{2}},\Delta_{k_{2}})\ar[d]_{f_{2}} \ar@{-->}[r]&\cdots \ar@{-->}[r]&(Y_{k_{i}},\Delta_{k_{i}})\ar[d]_{f_{i}}\ar@{-->}[r]&\cdots\\
[X,\omega]=:[X_{1},\omega_{1}] \ar@{-->}[r]& [X_{2},\omega_{2}] \ar@{-->}[r]&\cdots \ar@{-->}[r]&[X_{i},\omega_{i}]\ar@{-->}[r]&\cdots
}
$$
over $Z$, where all $Y_{k_{i}}$ and $X_{i}$ are projective over $Z$, such that 
\begin{itemize}
\item
$f_{i}\colon (Y_{k_{i}},\Delta_{k_{i}}) \to [X_{i},\omega_{i}]$ are quasi-log complex analytic spaces induced by normal pairs,  
\item
the sequence of upper horizontal maps is a sequence of steps of a $(K_{Y}+\Delta)$-MMP over $Z$ around $W$ with scaling of $f^{*}A$, 
\item
the lower horizontal sequence of maps is a sequence of steps of an $\omega$-MMP over $Z$ around $W$ with scaling of $A$, and
\item
if we put 
$$\lambda_{i}:={\rm inf}\{\mu \in \mathbb{R}_{\geq 0}\,| \,\text{$\omega_{i}+\mu A_{i}$ is nef over $W$} \}$$ 
for each $i \geq 1$, then the following properties hold.
\begin{itemize}
\item
$\lambda_{i}>\lambda_{i+1}$ for all $i \geq 1$, 
\item
$\omega_{i}+\lambda_{i-1} A_{i}$ is semi-ample over a neighborhood of $W$ for all $i \geq 1$, and  
\item
$\omega_{i}+t A_{i}$ is ample over a neighborhood of $W$ for all $i \geq 1$ and $t \in (\lambda_{i},\lambda_{i-1})$.  
\end{itemize}
\end{itemize}
Furthermore, if $Y$ is $\mathbb{Q}$-factorial over $W$, then $Y_{k_{i}}$ is also $\mathbb{Q}$-factorial over $W$ for every $i \geq 1$. 
\end{thm}

\begin{proof}
The argument from the proof of  \cite[Theorem 3.11]{has-mmp-normal-pair} works with no changes because we may use Theorem \ref{thm--mmpstep-quasi-log} and Theorem \ref{thm--abundance-quasi-log}. 
\end{proof}

\begin{cor}[cf.~{\cite[Corollary 3.12]{has-mmp-normal-pair}}]\label{cor--mmpwithscaling-normalpair}
Let $\pi \colon Y \to Z$ be a projective morphism from a normal analytic variety $Y$ to a Stein space $Z$, and let $W \subset Z$ be a compact subset such that $\pi$ and $W$ satisfy (P). 
Let $(Y,\Delta)$ be a normal pair such that $K_{Y}+\Delta$ is $\pi$-pseudo-effective and ${\rm NNef}(K_{Y}+\Delta/Z) \cap {\rm Nlc}(Y,\Delta)  \cap \pi^{-1}(W) = \emptyset$. 
Let $A$ be an effective $\mathbb{R}$-Cartier divisor on $Y$ such that $K_{Y}+\Delta+A$ is nef over $W$ and ${\rm Nlc}(Y, \Delta)={\rm Nlc}(Y, \Delta+A)$ 
set-theoretically. 
Then there exists a sequence of steps of a $(K_{Y}+\Delta)$-MMP over $Z$ around $W$ with scaling of $A$ 
$$(Y_{1},\Delta_{1}) \dashrightarrow \cdots \dashrightarrow (Y_{i},\Delta_{i}) \dashrightarrow (Y_{i+1},\Delta_{i+1})  \dashrightarrow\cdots$$
such that the non-biholomorphic locus of the MMP is disjoint from ${\rm Nlc}(Y,\Delta)$. 
Furthermore, if $Y$ is $\mathbb{Q}$-factorial over $W$, then $Y_{i}$ is also $\mathbb{Q}$-factorial over $W$ for all $i \geq 1$. 
\end{cor}

\begin{proof}
This follows from Lemma \ref{lem--extremal-ray}, \cite[Theorem 9.2]{fujino-quasi-log-analytic}, and Theorem \ref{thm--mmpstep-quasi-log}. 
\end{proof}

\begin{cor}[cf.~{\cite[Corollary 3.13]{has-mmp-normal-pair}}]\label{cor--mmp-nomralpair-Qfacdlt}
Let $\pi \colon Y \to Z$ be a projective morphism from a normal analytic variety $Y$ to a Stein space $Z$, and let $W \subset Z$ be a compact subset such that $\pi$ and $W$ satisfy (P). 
Let $(Y,\Delta)$ be a normal pair such that $K_{Y}+\Delta$ is $\pi$-pseudo-effective and ${\rm NNef}(K_{Y}+\Delta/Z) \cap {\rm Nlc}(Y,\Delta) \cap \pi^{-1}(W) = \emptyset$. 
Let $A$ be an $\mathbb{R}$-Cartier divisor on $Y$ such that $K_{Y}+\Delta+\lambda_{0} A$ is ample over a neighborhood of $W$ for some positive real number $\lambda_{0}$. 
Then, after shrinking $Z$ around $W$, there exists a sequence of steps of a $(K_{Y}+\Delta)$-MMP over $Z$ around $W$ with scaling of $A$ 
$$(Y_{1},\Delta_{1}) \dashrightarrow \cdots \dashrightarrow (Y_{i},\Delta_{i}) \dashrightarrow (Y_{i+1},\Delta_{i+1})  \dashrightarrow\cdots$$
such that
\begin{itemize}
\item
the non-biholomorphic locus of the MMP is disjoint from ${\rm Nlc}(Y,\Delta)$, and
\item
if we put 
$$\lambda_{i}:={\rm inf}\{\mu \in \mathbb{R}_{\geq 0}\,| \,\text{$K_{Y_{i}}+\Delta_{i}+\mu A_{i}$ is nef over $W$} \}$$  
for each $i \geq 1$ and $\lambda:={\rm lim}_{i \to \infty}\lambda_{i}$, then the following properties hold. 
\begin{itemize}
\item
The MMP terminates after finitely many steps or otherwise $\lambda \neq \lambda_{i}$ for every $i \geq 1$, and 
\item
$K_{Y_{i}}+\Delta_{i}+t A_{i}$ is semi-ample over a neighborhood of $W$ for all $i \geq 1$ and any $t \in (\lambda_{i},\lambda_{i-1}]$. 
\end{itemize}
\end{itemize}
Furthermore, if $Y$ is $\mathbb{Q}$-factorial over $W$, then $Y_{i}$ is also $\mathbb{Q}$-factorial over $W$ for every $i \geq 1$. 
\end{cor}

\begin{proof}
The argument from the proof of \cite[Corollary 3.13]{has-mmp-normal-pair} works with no changes. 
\end{proof}

\begin{rem}[cf.~{\cite[Remark 3.15]{has-mmp-normal-pair}}]
By the standard argument, we can check that the following properties of a normal pair $(Y,\Delta)$ are preserved under the $(K_{Y}+\Delta)$-MMP in Corollary \ref{cor--mmpwithscaling-normalpair} and Corollary \ref{cor--mmp-nomralpair-Qfacdlt}. 
\begin{itemize} \item $Y$ is $\mathbb{Q}$-factorial over $W$, \item $(Y,0)$ is $\mathbb{Q}$-factorial klt, and \item $(Y,\Delta^{<1}+{\rm Supp}\,\Delta^{\geq 1})$ is dlt.  \end{itemize}
\end{rem}

\begin{rem}[cf.~{\cite[Question 1.1]{has-mmp-normal-pair}, \cite[Remark 3.16]{has-mmp-normal-pair}}]\label{rem--mmp-reduction}
As in the algebraic case, we may consider the following question: Let $\pi \colon X \to Z$ be a projective morphism from a normal analytic variety to a Stein space, and let $W \subset Z$ be a compact subset such that $\pi$ and $W$ satisfy (P). 
Let $(X,\Delta)$ be a normal pair such that $K_{X}+\Delta$ is $\pi$-pseudo-effective and ${\rm NNef}(K_{X}+\Delta/Z) \cap {\rm Nlc}(X,\Delta)  \cap \pi^{-1}(W) = \emptyset$. 
Is there then a finite sequence of steps of a $(K_{X}+\Delta)$-MMP over $Z$ around $W$
$$(X,\Delta)=:(X_{1},\Delta_{1}) \dashrightarrow\cdots \dashrightarrow  (X_{i},\Delta_{i}) \dashrightarrow \cdots \dashrightarrow (X_{m},\Delta_{m})$$ such that $K_{X_{m}}+\Delta_{m}$ is nef over $Z$?

The existence of a sequence of steps of a $(K_{X}+\Delta)$-MMP over $Z$ around $W$ was proved by Corollary \ref{cor--mmpwithscaling-normalpair} or Corollary \ref{cor--mmp-nomralpair-Qfacdlt}. 
By the same argument as the one presented in \cite[Remark 3.16]{has-mmp-normal-pair}, the question can be reduced to the termination of all MMP for all klt pairs in the analytic setting. 
\end{rem}

\section{Minimal model program along log canonical locus}\label{sec--mmp-lc}

In this section we prove the main result of this paper and corollaries. 

\subsection{Minimal model program along the Kawamata log terminal locus}\label{sec--mmp-klt}

In this subsection we list the complex analytic analogs of the main results in \cite[Section 4]{has-mmp-normal-pair}.

\begin{thm}[cf.~{\cite[Theorem 4.4]{has-mmp-normal-pair}}]\label{thm--from-nonvanish-to-minmodel}
Let $\pi \colon X \to Z$ be a projective morphism from a normal analytic variety $X$ to a Stein space $Z$. 
Let $W \subset Z$ be a compact subset such that $\pi$ and $W$ satisfy (P). 
Let $(X,\Delta)$ be a normal pair such that ${\rm Nklt}(X,\Delta) \neq \emptyset$, and let $A$ be an effective $\pi$-ample $\mathbb{R}$-divisor on $X$ such that 
$${\rm Bs}|K_{X}+\Delta+A/Z|_{\mathbb{R}} \cap {\rm Nklt}(X,\Delta+A)  \cap \pi^{-1}(W)= \emptyset.$$ 
Then, after shrinking $Z$ around $W$, $(X,\Delta+A)$ has a $\mathbb{Q}$-factorial good minimal model over $Z$ around $W$. 
\end{thm}

\begin{proof}
The proofs of \cite[Propositions 4.1, 4.2, 4.3, Theorem 4.4]{has-mmp-normal-pair} work in our situation because we may use the results in \cite{fujino-analytic-bchm} instead of those in \cite{bchm} and furthermore we may use the complex analytic analogs of \cite[Lemmas 3.2, 3.3, Proposition 3.4]{has-mmp-normal-pair} (see Remark \ref{rem--analytic-models}). 
\end{proof}

\begin{thm}[cf.~{\cite[Theorem 4.6]{has-mmp-normal-pair}}]\label{thm--from-minimodel-to-termi}
Let $\pi \colon X \to Z$ be a projective morphism from a normal analytic variety $X$ to a Stein space $Z$. 
Let $W \subset Z$ be a compact subset such that $\pi$ and $W$ satisfy (P). 
Let $(X,\Delta)$ be a normal pair such that ${\rm Nklt}(X,\Delta) \neq \emptyset$, and let $A$ be an effective $\pi$-ample $\mathbb{R}$-divisor on $X$ such that 
$${\rm Bs}|K_{X}+\Delta+A/Z|_{\mathbb{R}} \cap {\rm Nklt}(X,\Delta)  \cap \pi^{-1}(W) = \emptyset.$$  
Let 
$$(X_{1},\Delta_{1}+A_{1}) \dashrightarrow \cdots \dashrightarrow (X_{i},\Delta_{i}+A_{i}) \dashrightarrow \cdots$$
be a sequence of steps of a $(K_{X}+\Delta+A)$-MMP over $Z$ around $W$ with scaling of $A$ such that if we put $$\lambda_{i}:={\rm inf}\{\mu \in \mathbb{R}_{\geq 0}\,| \,\text{$(K_{X_{i}}+\Delta_{i}+A_{i})+\mu A_{i}$ is nef over $W$} \}$$ for each $i \geq 1$, then 
${\rm lim}_{i \to \infty} \lambda_{i}=0$.  
Then $\lambda_{m}=0$ for some $m$ and $(X_{m},\Delta_{m}+A_{m})$ is a good minimal model of $(X,\Delta+A)$ over $Z$ around $W$. 
\end{thm}

\begin{proof}
The argument from the proof of \cite[Theorem 4.6]{has-mmp-normal-pair} works with no changes. 
We note that the complex analytic analog of \cite[Corollary 4.5]{has-mmp-normal-pair} holds and we may use it to prove Theorem \ref{thm--from-minimodel-to-termi}. 
\end{proof}

\begin{thm}[cf.~{\cite[Theorem 4.10]{has-mmp-normal-pair}}]\label{thm--non-vanishing-kltmmp-main}
Let $\pi \colon X \to Z$ be a projective morphism from a normal analytic variety $X$ to a Stein space $Z$, and let $W \subset Z$ be a compact subset such that $\pi$ and $W$ satisfy (P). 
Let $(X,\Delta)$ be a normal pair and $A$ a $\pi$-ample $\mathbb{R}$-divisor on $X$ such that $K_{X}+\Delta+A$ is globally $\mathbb{R}$-Cartier and $\pi$-pseudo-effective. 
Suppose that ${\rm NNef}(K_{X}+\Delta+A/Z) \cap {\rm Nklt}(X,\Delta) \cap \pi^{-1}(W) = \emptyset$. 
Suppose in addition that $(K_{X}+\Delta+A)|_{{\rm Nklt}(X,\Delta)}$, which we think of as an $\mathbb{R}$-line bundle on ${\rm Nklt}(X,\Delta)$, is semi-ample over a neighborhood of $W$. 
Then we have ${\rm Bs}|K_{X}+\Delta+A/Z|_{\mathbb{R}} \cap {\rm Nklt}(X,\Delta) \cap \pi^{-1}(W) = \emptyset$. 
\end{thm}

\begin{proof}
The proofs of \cite[Propositions 4.8, 4.9, Theorem 4.10]{has-mmp-normal-pair} work in our situation because we may use Theorem \ref{thm--from-minimodel-to-termi} and Theorem \ref{thm--vanishing-quasi-log}. 
\end{proof}

\begin{thm}[cf.~{\cite[Theorem 4.11]{has-mmp-normal-pair}}]\label{thm--klt-minmodeltheory-main}
Let $\pi \colon X \to Z$ be a projective morphism from a normal analytic variety $X$ to a Stein space $Z$, and let $W \subset Z$ be a compact subset such that $\pi$ and $W$ satisfy (P). 
Let $(X,\Delta)$ be a normal pair and $A$ a $\pi$-ample $\mathbb{R}$-divisor on $X$ such that $K_{X}+\Delta+A$ is globally $\mathbb{R}$-Cartier and $\pi$-pseudo-effective. 
Suppose that ${\rm NNef}(K_{X}+\Delta+A/Z) \cap {\rm Nklt}(X,\Delta) \cap \pi^{-1}(W)  = \emptyset$. 
Suppose in addition that $(K_{X}+\Delta+A)|_{{\rm Nklt}(X,\Delta)}$, which we think of as an $\mathbb{R}$-line bundle on ${\rm Nklt}(X,\Delta)$, is semi-ample over a neighborhood of $W$. 
Let 
$$(X_{1},\Delta_{1}+A_{1}) \dashrightarrow \cdots \dashrightarrow (X_{i},\Delta_{i}+A_{i}) \dashrightarrow \cdots$$
be a sequence of steps of a $(K_{X}+\Delta+A)$-MMP over $Z$ around $W$ with scaling of $A$.  We put 
$$\lambda_{i}:={\rm inf}\{\mu \in \mathbb{R}_{\geq 0}\,| \,\text{$(K_{X_{i}}+\Delta_{i}+A_{i})+\mu A_{i}$ is nef over $W$} \}$$ for each $i \geq 1$. 
Then 
${\rm lim}_{i \to \infty} \lambda_{i}=\lambda_{m}$ for some $m$. 
In particular, if ${\rm lim}_{i \to \infty} \lambda_{i}=0$ then the $(K_{X}+\Delta+A)$-MMP terminates with a good minimal model over $Z$ around $W$. 
\end{thm}

\begin{proof}
The theorem follows from Theorem \ref{thm--non-vanishing-kltmmp-main} and Theorem \ref{thm--from-minimodel-to-termi}. 
\end{proof}

\subsection{Proofs of the main result and corollaries}

The goal of this subsection is to prove Theorem \ref{thm--termination-lcmmp-main} and the corollaries in the introduction.

\begin{lem}[cf.~{\cite[Lemma 5.1]{has-mmp-normal-pair}}]\label{lem--mmp-fibration-kltlocus}
Let $f\colon (Y,\Delta) \to [X,\omega]$ be a quasi-log complex analytic space induced by a normal pair. 
Let $\pi \colon X \to Z$ be a projective morphism to a Stein space $Z$, and let $W \subset Z$ be a compact subset such that $\pi$ and $W$ satisfy (P). 
Let $A$ be a $\pi$-ample $\mathbb{R}$-divisor on $X$ such that $\omega+A$ is $\pi$-pseudo-effective. 
Suppose that ${\rm NNef}(\omega+A/Z) \cap{\rm Nqklt}(X,\omega) \cap \pi^{-1}(W) = \emptyset$. 
Suppose in addition that $(\omega+A)|_{{\rm Nqklt}(X,\omega)}$, which we think of as an $\mathbb{R}$-line bundle on ${\rm Nqklt}(X,\omega)$, is semi-ample over a neighborhood of $W$. 
Let
$$X_{1} \dashrightarrow \cdots \dashrightarrow X_{i} \dashrightarrow \cdots$$
be a sequence of steps of an $(\omega+A)$-MMP over $Z$ around $W$ with scaling of $A$ such that if we put $$\lambda_{i}:={\rm inf}\{\mu \in \mathbb{R}_{\geq 0}\,| \,\text{$\omega_{i}+A_{i}+\mu A_{i}$ is nef over $W$} \}$$ for each $i \geq 1$, then 
${\rm lim}_{i \to \infty} \lambda_{i}=0$.  
Then $\lambda_{m}=0$ for some $m$. 
\end{lem}

\begin{proof}
We can apply the argument from the proof of \cite[Lemma 5.1]{has-mmp-normal-pair} to our situation since we can use Lemma \ref{lem--can-bundle-formula} and Theorem \ref{thm--klt-minmodeltheory-main}.  
\end{proof}

\begin{thm}[cf.~{\cite[Theorem 5.2]{has-mmp-normal-pair}}]\label{thm--mmp-lclocus-main}
Let $f\colon (Y,\Delta) \to [X,\omega]$ be a quasi-log complex analytic space induced by a normal pair. 
Let $\pi \colon X \to Z$ be a  projective morphism to a Stein space $Z$, and let $W \subset Z$ be a compact subset such that $\pi$ and $W$ satisfy (P). 
Let $A$ be a $\pi$-ample $\mathbb{R}$-divisor on $X$ such that $\omega+A$ is globally $\mathbb{R}$-Cartier and $\pi$-pseudo-effective. 
Suppose that ${\rm NNef}(\omega+A/Z) \cap{\rm Nqlc}(X,\omega) \cap \pi^{-1}(W)  = \emptyset$. 
Suppose in addition that $(\omega+A)|_{{\rm Nqlc}(X,\omega)}$, which we think of as an $\mathbb{R}$-line bundle on ${\rm Nqlc}(X,\omega)$, is semi-ample over a neighborhood of $W$. 
Then, after shrinking $Z$ around $W$, there is a sequence of steps of an $(\omega+A)$-MMP over $Z$ around $W$ with scaling of $A$ 
$$X_{1} \dashrightarrow  X_{2} \dashrightarrow \cdots \dashrightarrow X_{n},$$
represented by bimeromorphic contractions such that $\omega_{n}+A_{n}$ is semi-ample over $Z$, where $\omega_{n}$ and $A_{n}$ are the strict transforms of $\omega$ and $A$ on $X_{n}$ respectively.
\end{thm}

\begin{proof}
The argument from the proof of \cite[Theorem 5.2]{has-mmp-normal-pair} works with no changes. 
We note that we need the complex analytic analog of the special termination of the MMP (\cite{fujino-sp-ter}), adjunction for quasi-log schemes (\cite[Theorem 6.1.2 (i)]{fujino-book}), and the base point free theorem for $\mathbb{R}$-Cartier divisors on quasi-log schemes (\cite[Theorem 2.21]{has-mmp-normal-pair}). 
These are known as given in \cite[Subsection 3.6]{eh-analytic-mmp}, \cite[Theorem 4.4]{fujino-quasi-log-analytic}, and Theorem \ref{thm--abundance-quasi-log}. 
\end{proof}

\begin{thm}[cf.~{\cite[Theorem 5.3]{has-mmp-normal-pair}}]\label{thm--termination-lcmmp-main}
Let $\pi \colon X \to Z$ be a projective morphism from a normal analytic variety $X$ to a Stein space $Z$, and let $W \subset Z$ be a compact subset such that $\pi$ and $W$ satisfy (P). 
Let $(X,\Delta)$ be a normal pair and $A$ a $\pi$-ample $\mathbb{R}$-divisor on $X$ such that $K_{X}+\Delta+A$ is globally $\mathbb{R}$-Cartier and $\pi$-pseudo-effective. 
Suppose that ${\rm NNef}(K_{X}+\Delta+A/Z) \cap {\rm Nlc}(X,\Delta) \cap \pi^{-1}(W)  = \emptyset$. 
Suppose in addition that $(K_{X}+\Delta+A)|_{{\rm Nlc}(X,\Delta)}$, which we think of as an $\mathbb{R}$-line bundle on ${\rm Nlc}(X,\Delta)$, is semi-ample over a neighborhood of $W$. 
We put $(X_{1},B_{1}):=(X,\Delta+A)$. 
Let $H$ be a $\pi$-ample $\mathbb{R}$-divisor on $X$. 
Then, after shrinking $Z$ around $W$, there exists a sequence of steps of a $(K_{X_{1}}+B_{1})$-MMP over $Z$ around $W$ with scaling of $H$ 
$$
\xymatrix{
(X_{1},B_{1})\ar@{-->}[r]&\cdots \ar@{-->}[r]& (X_{i},B_{i})\ar@{-->}[rr]\ar[dr]&& (X_{i+1},B_{i+1})\ar[dl] \ar@{-->}[r]&\cdots\ar@{-->}[r]&(X_{m},B_{m}) \\
&&&V_{i}
}
$$
represented by bimeromorphic contractions, where $(X_{i},B_{i}) \to V_{i} \leftarrow (X_{i+1},B_{i+1})$ is a step of a $(K_{X_{i}}+B_{i})$-MMP over $Z$ around $W$, such that
\begin{itemize}
\item
the non-biholomorphic locus of the MMP is disjoint from ${\rm Nlc}(X,\Delta)$, 
\item
$\rho(X_{i}/Z; W) - \rho(V_{i}/Z; W)=1$ for every $i \geq 1$, and
\item
$K_{X_{m}}+B_{m}$ is semi-ample over $Z$. 
\end{itemize}
Moreover, if $X$ is $\mathbb{Q}$-factorial over $W$, then all $X_{i}$ are also $\mathbb{Q}$-factorial over $W$. 
\end{thm}

\begin{proof}
The argument from the proof of \cite[Theorem 5.3]{has-mmp-normal-pair} works with no changes. 
\end{proof}

\begin{thm}[cf.~{\cite[Theorem 5.4]{has-mmp-normal-pair}}]\label{thm--termination-all-lcmmp}
Let $\pi \colon X \to Z$ be a projective morphism from a normal analytic variety $X$ to a Stein space $Z$, and let $W \subset Z$ be a compact subset such that $\pi$ and $W$ satisfy (P). 
Let $(X,\Delta)$ be a normal pair and $A$ a $\pi$-ample $\mathbb{R}$-divisor on $X$ such that $K_{X}+\Delta+A$ is globally $\mathbb{R}$-Cartier and $\pi$-pseudo-effective. 
Suppose that ${\rm NNef}(K_{X}+\Delta+A/Z) \cap {\rm Nlc}(X,\Delta) \cap \pi^{-1}(W) = \emptyset$. 
Suppose in addition that $(K_{X}+\Delta+A)|_{{\rm Nlc}(X,\Delta)}$, which we think of as an $\mathbb{R}$-line bundle on ${\rm Nlc}(X,\Delta)$, is semi-ample over a neighborhood of $W$. 
Let 
$$(X,\Delta+A)=:(X_{1},\Delta_{1}+A_{1}) \dashrightarrow \cdots \dashrightarrow (X_{i},\Delta_{i}+A_{i}) \dashrightarrow \cdots$$
be a sequence of steps of a $(K_{X}+\Delta+A)$-MMP over $Z$ around $W$ with scaling of $A$ such that if we put 
$$\lambda_{i}:={\rm inf}\{\mu \in \mathbb{R}_{\geq 0}\,| \,\text{$(K_{X_{i}}+\Delta_{i}+A_{i})+\mu A_{i}$ is nef over $W$} \}$$
for each $i \geq 1$, then ${\rm lim}_{i \to \infty} \lambda_{i}=0$. 
Then $\lambda_{m}=0$ for some $m$ and $K_{X_{m}}+\Delta_{m}+A_{m}$ is semi-ample over a neighborhood of $W$. 
\end{thm}

\begin{proof}
The argument from the proof of \cite[Theorem 5.3]{has-mmp-normal-pair} works with no changes. 
\end{proof}

Now we are ready to prove the corollaries in the introduction.

\begin{proof}[Proof of Corollary \ref{cor--nonvan-intro}]
This immediately follows from Theorem \ref{thm--termination-lcmmp-main}. 
\end{proof}

\begin{proof}[Proof of Corollary \ref{cor--mmp-from-nonvan-intro}]
This immediately follows from Theorem \ref{thm--termination-lcmmp-main}. 
\end{proof}

\begin{proof}[Proof of Corollary \ref{cor--mmp-lc-strictnefthreshold-intro}]  
This follows from Theorem \ref{thm--termination-all-lcmmp}. 
\end{proof}

\begin{proof}[Proof of Corollary \ref{cor--finite-generation-intro}]
We only need to check the finite generation on a Stein open neighborhood of any point $z \in Z$. 
Then the finite generation follows from Theorem \ref{thm--termination-lcmmp-main}.  
\end{proof}

\section{Minimal model program without shrinking base space}\label{sec--glue}

In this section, we study the MMP for normal pairs for algebraic stacks and analytic stacks. 
See \cite[Section 27]{lyu--murayama-mmp} and \cite[Section 3]{eh-analytic-mmp-2} for the lc case. 

We will freely use the definitions in \cite[Section 3]{eh-analytic-mmp-2}. 
In particular, all algebraic stacks are assumed to be locally of finite type over a base field $k$. 
{\em Complex analytic stacks} are defined as stacks $X$ over the category of complex analytic spaces equipped with the \'{e}tale topology such that
the diagonal map $\Delta_{X}$ is representable by complex analytic spaces and there exists a smooth covering $U\to X$ from a complex analytic space $U$.
All arguments in this section are conducted within the category of schemes or complex analytic spaces. The term {\em stacks} refers to either algebraic stacks or complex analytic stacks,
and {\em spaces} refers to algebraic spaces or complex analytic spaces unless otherwise stated. 

The morphism $\pi \colon X \to Z$ of stacks is {\em projective} if there exists an ample $\mathbb{R}$-line bundle $\mathcal{L}$ on $X$ over $Z$ in the sense of \cite[Definition 3.7]{eh-analytic-mmp-2}. 
We note that any projective morphism is representable by spaces.

\begin{thm}\label{thm--mmp-normalpair-another}
Let $\pi \colon X \to Z$ be a projective morphism from a normal analytic variety $X$ to a Stein space $Z$. 
Let $(X,\Delta)$ be a normal pair and $A$ a $\pi$-ample $\mathbb{R}$-divisor on $X$ such that $K_{X}+\Delta+A$ is globally $\mathbb{R}$-Cartier and $\pi$-pseudo-effective. 
Suppose that ${\rm NNef}(K_{X}+\Delta+A/Z) \cap {\rm Nlc}(X,\Delta)  = \emptyset$. 
Suppose in addition that $(K_{X}+\Delta+A)|_{{\rm Nlc}(X,\Delta)}$, which we think of as an $\mathbb{R}$-line bundle on ${\rm Nlc}(X,\Delta)$, is semi-ample over a neighborhood of any compact subset of $Z$. 
We put $(X_{1},B_{1}):=(X,\Delta+A)$. 
Let $H$ be a $\pi$-ample $\mathbb{R}$-divisor on $X$. 
Then, for any point $z \in Z$, after shrinking $Z$ around $z$, there exists a sequence of steps of a $(K_{X_{1}}+B_{1})$-MMP over $Z$ around $z$ with scaling of $H$
$$
(X_{1},B_{1}) \dashrightarrow (X_{2},B_{2}) \dashrightarrow \cdots \dashrightarrow (X_{m},B_{m}),
$$
which is represented by bimeromorphic contractions, such that if we define $$\lambda_{i}:={\rm inf}\{\nu \in \mathbb{R}_{\geq 0}\,|\, \text{$K_{X_{i}}+B_{i}+\nu H_{i}$ {\rm is nef over} $z$}\}$$
for each $i \geq 1$, then the following holds.
\begin{itemize}
\item
$\lambda_{i}>\lambda_{i+1}$ for each $i \geq 1$ and $\lambda_{m}=0$, 
\item
$K_{X_{i}}+B_{i}+ \lambda_{i}H_{i}$ is semi-ample over $Z$ for any $i \geq 1$, in particular, $(X_{m},B_{m})$ is a good minimal model of $(X_{1},B_{1})$ over $Z$ around $z$, and
\item
for any $i \geq 2$ and $t \in (\lambda_{i},\lambda_{i-1})$, the bimeromorphic contraction 
$X_{1} \dashrightarrow X_{i}$ over $Z$ is the ample model of $K_{X_{1}}+B_{1}+t H_{1}$ over $Z$ {\rm (cf.~{\cite[Definition~3.6.5]{bchm}})}.
\end{itemize}
\end{thm}

\begin{proof}
By shrinking $Z$, we may assume that $A-\epsilon H$ is ample over $Z$. 
We take a general member $A'$ of $|A-\epsilon H/Z|_{\mathbb{R}}$ such that ${\rm Nlc}(X,\Delta)={\rm Nlc}(X,\Delta+A')$ as closed analytic subspace of $X$. 
By replacing $\Delta$, $A$, and $H$ with $\Delta+A'$, $\epsilon H$, and $\epsilon H$ respectively, we may assume $A=H$. We fix a point $z \in Z$. 
By Theorem \ref{thm--mmp-nefthreshold-strict}, there exists a sequence of steps of a $(K_{X}+\Delta+A)$-MMP over $Z$ around $z$ with scaling of $A$ 
$$
(X_{1},\Delta_{1}+A_{1}) \dashrightarrow (X_{2},\Delta_{2}+A_{2}) \dashrightarrow \cdots \dashrightarrow (X_{i},\Delta_{i}+A_{i}) \dashrightarrow \cdots
$$
such that if we put 
$$\lambda_{i}:={\rm inf}\{\mu \in \mathbb{R}_{\geq 0}\,| \,\text{$K_{X_{i}}+\Delta_{i}+A_{i}+\mu A_{i}$ is nef over $z$} \}$$ 
for each $i \geq 1$, then the following properties hold.
\begin{itemize}
\item
$\lambda_{i}>\lambda_{i+1}$ for all $i \geq 1$, and 
\item
$K_{X_{i}}+\Delta_{i}+A_{i}+t A_{i}$ is ample over a neighborhood of $z$ for all $i \geq 1$ and $t \in (\lambda_{i},\lambda_{i-1})$.  
\end{itemize}
We put $\lambda:={\rm lim}_{i \to \infty} \lambda_{i}$. 
Then the MMP is a $(K_{X}+\Delta+(1+\lambda)A)$-MMP over $Z$ around $z$ with scaling of $A$. 
By Theorem \ref{thm--termination-all-lcmmp}, we have $\lambda=0$ and $\lambda_{m}=\lambda$ for some $m \geq 1$. 
By shrinking $Z$ around $z$, we get the desired MMP. 
\end{proof}

\begin{defn}[$t$-th output of $D$-MMP with scaling of $H$, {\cite[Definition 3.13]{eh-analytic-mmp-2}}] \label{outputMMP}
Let $\pi\colon X\to Z$ be a projective morphism of stacks
and assume that $X$ is normal.
Let $D$ and $H$ be $\mathbb{R}$-line bundles on $X$ over $Z$ (\cite[Definition 3.7]{eh-analytic-mmp-2}) such that $H$ is big over $Z$.
Let us assume that the pseudo-effective threshold of $H|_{U}$ with respect to $D|_{U}$, denoted by $\mu$, is constant for any smooth morphism $U\to Z$ from a space $U$ such that $X\times_{Z}U$ is non-empty.
Note that if $X$ is irreducible, then 
this assumption is satisfied 
and $\mu$ equals
the minimal non-negative number such 
that $D+\mu H$ is pseudo-effective over $Z$.
For $t>\mu$, 
a birational or bimeromorphic map $\varphi_{t}\colon X\dasharrow X_{t}$
is called 
a {\em $t$-th output of a $D$-MMP with scaling of $H$ over $Z$}
if it is an ample model (\cite[Definition 3.10]{eh-analytic-mmp-2}) of $D+(t-\varepsilon) H$ for sufficiently small $\varepsilon >0$ smooth locally on $Z$, that is, there exist a smooth covering $\{Z_{i}\to Z\}_{i}$ and positive numbers $\{a_{i}\}_{i}$ such that the base change of $\varphi_{t}$ to $Z_{i}$ is an ample model of $D|_{Z_i}+(t-\varepsilon)H|_{Z_i}$ for $0<\varepsilon<a_i$.

We say that the {\em $D$-MMP with scaling of $H$ over $Z$ exists} if there exists a $t$-th output of a $D$-MMP with scaling of $H$ over $Z$ for any $t>\mu$.
\end{defn}

For the termination of the above MMP, see \cite[Definition 3.17]{eh-analytic-mmp-2}.

\begin{thm}\label{thm-glue}
Let $\pi \colon X \to Z$ be a projective morphism of algebraic stacks locally of finite type over $\mathbb{C}$ or complex analytic stacks.
Let $(X,B)$ be an irreducible normal pair such that $K_{X}+B$ is pseudo-effective over $Z$. 
Suppose that there exists a smooth covering $\{Z_{i} \to Z\}_{i}$ from quasi-projective schemes or Stein spaces $Z_{i}$ such that
\begin{itemize}
\item
the pullback $K_{X_{i}}+B_{i}$ is globally $\mathbb{R}$-Cartier for any $i$, 
\item
$B_{i}=\Delta_{i}+A_{i}$ for some effective $\mathbb{R}$-divisor $\Delta_{i}$ and effective relatively ample $\mathbb{R}$-divisor $A_{i}$ on $X_{i}$, 
\item
${\rm NNef}(K_{X_{i}}+B_{i}/Z_{i}) \cap {\rm Nlc}(X_{i},\Delta_{i})  = \emptyset$, and
\item 
the $\mathbb{R}$-line bundle $(K_{X_{i}}+B_{i})|_{{\rm Nlc}(X_{i},\Delta_{i})}$ on ${\rm Nlc}(X_{i},\Delta_{i})$ is semi-ample over $Z_{i}$. 
\end{itemize}
Then a $(K_{X}+B)$-MMP with scaling of a $\pi$-ample $\mathbb{R}$-line bundle $H$ over $Z$ in the sense of \cite[Definition 3.13]{eh-analytic-mmp-2} exists and it terminates smooth locally on $Z$ whose output is a good minimal model of $(X,B)$ over $Z$. 
\end{thm}

\begin{proof}
This follows from Theorem \ref{thm--mmp-normalpair-another} and \cite[Lemma 3.14]{eh-analytic-mmp-2}. 
\end{proof}


\end{document}